\documentclass{article}

\usepackage{geometry}
\usepackage[T1]{fontenc}
\usepackage[utf8]{inputenc}
\usepackage{lmodern,amsmath,amssymb,systeme,amsthm,amsfonts,graphicx,bbm,cite,needspace,bm}
\usepackage[inline]{enumitem}
\usepackage[breaklinks]{hyperref}

\numberwithin{equation}{section}

\newtheorem{theorem}{Theorem}[section]
\newtheorem{lemma}[theorem]{Lemma}
\newtheorem{claim}[theorem]{Claim}

\newtheorem{proposition}[theorem]{Proposition}

\theoremstyle{definition}
\newtheorem{remark}[theorem]{Remark}
\newtheorem{notation}[theorem]{Notation}
\newtheorem{definition}[theorem]{Definition}

\newcommand{\pr}{\mathbb{P}}
\newcommand{\ev}{\mathbb{E}}
\newcommand{\indic}[1]{\mathbbm{1}{\left[\,{#1}\,\right]}}
\newcommand{\norm}[1]{\left\lVert#1\right\rVert}

\newcommand{\Pspace}{\mathcal{P}_{\Delta}}
\newcommand{\nonreg}{\mathcal{P}_\Delta^\circ}
\newcommand{\nearreg}[1]{\mathcal{P}_{\Delta,\, #1}^\circ}
\newcommand{\Rspace}{\mathcal{R}_{\Delta}}
\newcommand{\pfromr}{\underline{p}(\varepsilon,\underline{r})}
\newcommand{\tunder}{\underline{t}_{\underline{p}, \delta}}
\newcommand{\tover}{\overline{t}_{\underline{p}, \delta}}
\newcommand{\toverone}{\overline{t}_{\underline{p}, 1}}
\newcommand{\wunder}{\underline{w}_{\underline{p}, \delta}}
\newcommand{\wover}{\overline{w}_{\underline{p}, \delta}}
\newcommand{\werror}{\omega_{\underline{r}}}
\newcommand{\Toperator}[1]{T_{\underline{p},#1}}
\newcommand{\lamp}{\lambda_{\underline{p}}}
\newcommand{\lamunder}[1]{\underline{\lambda}_{\underline{p},#1}}
\newcommand{\lamover}[1]{\overline{\lambda}_{\underline{p},#1}}
\newcommand{\lamerror}{\Lambda_{\underline{r}}}
\newcommand{\tauunder}[1]{\tau^-_{\underline{p},\,#1}}
\newcommand{\tauover}[1]{\tau^+_{\underline{p},\,#1}}
\newcommand{\Hp}{H_{\underline{p}}}
\newcommand{\Herror}{\mathcal{H}_{\underline{r}}}
\newcommand{\Hint}{I_{\underline{p}}}
\newcommand{\Herrorint}{\mathcal{I}_{\underline{r}}}
\newcommand{\Hintlam}{\widetilde{I}_{\underline{p}}}
\newcommand{\Hintunder}[1]{\underline{I}_{\underline{p},\, #1}}
\newcommand{\Hintover}[1]{\overline{I}_{\underline{p},\, #1}}
\newcommand{\tc}{\hat{t}_c(\underline{p})}
\newcommand{\diffunder}{\underline{\phi}_{\underline{p}, \delta}}
\newcommand{\diffover}{\overline{\phi}_{\underline{p}, \delta}}
\newcommand{\werrorfirst}{\alpha_{\underline{r}}}
\newcommand{\werrorsecond}{\beta_{\underline{r}}}
\newcommand\puteqnum{\hfill\refstepcounter{equation}\textup{(\theequation)}}

\title{Critical time of the almost 2-regular\\random degree constrained process}

\author{Bal\'azs R\'ath \thanks{Department of Stochastics, Institute of Mathematics, Budapest University of Technology and Economics,
M\H{u}egyetem rkp.\ 3., H-1111 Budapest, Hungary; HUN-REN--BME Stochastics Research Group, Budapest University of Technology and Economics, M\H{u}egyetem rkp.\ 3., H-1111 Budapest, Hungary;
HUN-REN Alfr\'ed R\'enyi Institute of Mathematics, Re\'altanoda utca 13-15, H-1053 Budapest,
Hungary. Email:~\href{mailto:rathb@math.bme.hu}{\tt rathb@math.bme.hu}, \href{mailto:szokemp@math.bme.hu}{\tt szokemp@math.bme.hu}.}
\and
M\'arton Sz\H{o}ke \footnotemark[1]}

\begin{document}

\maketitle

\begin{abstract} 
We study the phase transition of the random degree constrained process (RDCP), a time-evolving random graph model introduced by Ruci\'nski and Wormald \cite{RW92} that generalizes the random $d$-process to the non-regular setting: each vertex of the complete graph $K_n$ has its pre-assigned degree constraint (i.e., a number from the set $\{2,\dots,\Delta \}$), we attempt to add the edges one-by-one in a uniform random order, but a new edge is added only if it does not violate the degree constraints at its end-vertices. Warnke and Wormald~\cite{WW} identified the critical time of the RDCP when the giant component emerges as~${n \to \infty}$. R\'ath, Sz\H{o}ke and Warnke~\cite{RSW} identified the local weak limit of the RDCP and gave an alternative characterization of the critical time in terms of the principal eigenvalue of the branching operator of the multi-type branching process that arises as the local limit object. 

In the current paper we use this spectral characterization to study the critical time of the RDCP in the almost $2$-regular case, i.e., when the degree constraint of most of the vertices is equal to $2$. In this case the giant component emerges quite late, and our main result provides the precise asymptotics of the critical time as the model approaches $2$-regularity.
Interestingly, our formula asymptotically matches the well-known Molloy--Reed formula~\cite{MR}, despite the fact that Molloy, Surya and Warnke~\cite{MSW} proved that the final graph of the RDCP is not contiguous to the configuration model with the same degree sequence.

\vspace{0.5em}
\noindent
\textbf{Keywords:} random degree constrained process; giant component phase transition; critical time; Perron--Frobenius theory; ordinary differential equations

\noindent
\textbf{AMS MSC 2020:} 05C80; 15B48; 82B26; 34A30
\end{abstract}

\section{Introduction}\label{sec:intro}

\textbf{Giant and gelation.} Emergence of the giant component in dynamic networks and gelation in the kinetic theory of polymers have a lot in common. Both phenomena are examples of connectivity/percolation phase transitions, but the rigorous theory of random graphs tends to use combinatorial or algorithmic tools (e.g.\ exploration processes) while the analytic theory of chemistry tends to use differential equation methods.
A survey of these two approaches is given in~\cite{A99}. Let us recall three famous examples that have been studied using both approaches: 
\begin{itemize}
\item The fluid limit of the densities of connected components of size $1,2,3,\dots$ in the dynamic Erd\H{o}s--R\'enyi graph~\cite{ER60} is exactly the solution to the system of differential equations known as Flory's equation \cite{F41} by chemists and Smoluchowski's coagulation equation with multiplicative kernel~\cite{S16} by physicists.
\item The canonical random (multi)graph with a given degree sequence is the configuration model \cite{B80}: we pick a uniformly random perfect matching on the collection of half-edges that arise from the prescribed degree sequence.
The dynamic version of the configuration model is equivalent to the chemical system where each monomer has a predefined functionality, and monomers connect to each other at a rate proportional to the product of the number of their unused functionalities: the resulting Flory-type system of differential equations was studied e.g.\ in~\cite{K18}, analytically recovering the explicit formula for the percolation threshold of the configuration model (i.e., the Molloy--Reed formula~\cite{MR}). 
\item The random $d$-process, introduced in~\cite{RW92}, is a dynamical (simple) random graph model in which a vertex saturates when its degree reaches~$d$, and at each step we add an edge that we uniformly choose from the set of (not yet added) edges that still have two unsaturated end-vertices. The Flory-type system of differential equations that describes the phase transition of the $d$-process was studied in~\cite{BK}.
\end{itemize}
Our current paper studies the phase transition of the \emph{random degree constrained process (RDCP)}, a generalization of the random $d$-process, using the \emph{analytic approach}.

\smallskip

\textbf{RDCP.} The RDCP was in fact already introduced in~\cite{RW92}: 
we attempt to add the edges of the complete graph $K_n$ one-by-one in a uniform random order, but a new edge is added only if it does not violate the pre-assigned degree constraints at its end-vertices. We say that a vertex saturates when its degree reaches its pre-assigned degree constraint. In~\cite{RW92} it is shown that all vertices in the final graph of the RDCP are saturated with high probability as ${n \to \infty}$.
The random $d$-process is the special case of the RDCP in which every vertex has the same fixed degree constraint~${d \in \mathbb{N}_+}$.

\smallskip

\textbf{Difference between the configuration model and the RDCP.} The configuration model and the final graph of the RDCP are two different constructions that achieve the same goal: to produce a random graph with a given degree sequence. It is proved in~\cite{MSW} that if the asymptotic degree distribution is not concentrated on a single value (i.e., if the graph is not nearly regular), then there exists an event~$A_n$ (defined in terms of neighborhood statistics) such that the probability of~$A_n$ in the configuration model goes to~$1$ as~${n \to \infty}$, while the probability that the final graph of the RDCP with the same degree sequence satisfies the same~$A_n$ goes to~$0$ as $n \to \infty$. In plain words, the two random graph models are genuinely different.

\smallskip 

\textbf{Phase transition of the RDCP.} The emergence of the giant component in the RDCP was proved in~\cite{WW}. An analytic characterization of the critical time of the random $d$-process (i.e., the time when the giant emerges) was given in~\cite{BK}, and~\cite{WW} gave a different analytic characterization which extends to the case of the RDCP. Unlike the Molloy--Reed formula~\cite{MR} for the configuration model, these characterizations of the critical time of the RDCP are not at all explicit (but they produce the same numerical approximations of the critical time of the random $d$-process up to a reasonable precision). The method of~\cite{BK} is to study the non-linear PDE that arises from the Flory-type equations if one takes the generating function and to define the critical time as the time when conservation of mass contained in small components breaks down. The method of~\cite{WW} is to write a closed system of differential equations for quantities related to susceptibility and define the critical time as the time when the susceptibility blows up.

\smallskip 

\textbf{Spectral characterization of the critical time.}
In our current paper, we use a third alternative analytic approach to study the critical time of the RDCP. This approach does not give an explicit formula either, but it allows us to derive new and interesting asymptotic formulas for the critical time.
Our approach relies on the notion of local weak convergence of graph sequences, also known as Benjamini--Schramm convergence. Formally, local weak convergence in probability is equivalent to the concentration of all local neighborhood statistics as $n \to \infty$, cf.~\cite[Section~2]{H24b}. In the case of the Erd\H{o}s--R\'enyi graph as well as the configuration model, the local weak limit is the family tree of a Bienaym\'e--Galton--Watson branching process, and thus the critical time is the time when the expected number of offspring of this branching process is equal to~$1$ (this approach easily recovers the Molloy--Reed formula). Together with Warnke, in \cite{RSW} we identify the local weak limit of the RDCP at any intermediate time as the family tree of a \emph{multi-type branching process} with type space~$\mathbb{R}_+$, and characterize the critical time as the time when the principal eigenvalue of the associated branching operator is equal to $1$. The branching operator of the RDCP is an integral operator that can be studied using a combination of Perron--Frobenius and Sturm--Liouville theories.

\smallskip 

\textbf{Almost $2$-regular RDCP.} We use this approach to study the critical time of the RDCP in the almost $2$-regular case, i.e., when the fraction of vertices with degree constraint~$2$ is ${1-\varepsilon}$, and the degree constraints of the rest of the vertices are larger than~$2$. Structurally, this implies that the final graph has approximately~${\varepsilon \cdot n}$~\emph{hubs} (i.e., vertices of degree at least~$3$) connected by long simple paths. Thus, if one removes even a small fraction of its edges, it easily becomes disconnected. In other words, the critical time of this graph is rather big.
We provide an asymptotic formula for the critical time in terms of the perturbative parameter~$\varepsilon$ and the relative degree distribution of the hubs. Interestingly, our asymptotic formula matches the naive guess that one obtains from the Molloy--Reed formula in the $\varepsilon \to 0$ limit, despite the fact that for any fixed $\varepsilon>0$ the final graph of the RDCP is not contiguous to the configuration model (as proved in~\cite{MSW}) and edges of the final graph of the RDCP arrive in a correlated way (and thus naively modeling the RDCP at an intermediate time with Bernoulli bond percolation on the final graph is incorrect).

\smallskip 

\textbf{A conjecture about the regular case.} Our asymptotic result suggests that the differences between the RDCP and the configuration model vanish in the almost $2$-regular case as $\varepsilon \to 0$, hinting at the universality of the structure of almost $2$-regular / barely supercritical / near-critical versions of these random graph models. However, let us note that even in the degenerate $\varepsilon =0$ case there are open questions about the similarity of these two random graph models: Wormald conjectured that the random $d$-regular graph (i.e., the special case of the configuration model when all degrees are equal to $d$) and the random $d$-process (i.e., the final graph of the RDCP when all degrees are equal to $d$) cannot be distinguished with high confidence (for the precise formulation, see \cite[Conjecture~6.3]{W99b}). However, note that the local weak limit of both of these random graph models is the $d$-regular infinite tree.

\smallskip 

\textbf{Two time parametrizations.} Recall that one constructs the RDCP by attempting to add the edges of the complete graph $K_n$ one-by-one and an edge is added only if it does not violate the degree constraints.
Note that there are two conventions regarding the time parametrization of the RDCP in the literature:
\begin{enumerate}[label=(\roman*)]
\item\label{item:WW_parametrization} e.g.~\cite{WW} uses the continuous-time parameter~$s$ to describe the state of the RDCP graph when it has approximately ${n\cdot s}$ successfully deposited edges, while 
\item\label{item:BK_param} e.g.~\cite{BK} uses the continuous-time parameter~$t$ to describe the state of the RDCP graph after attempting to deposit approximately ${n\cdot t/2}$ edges.
\end{enumerate}
We will state our main result using both of the above time parametrizations, but in our proof we will use convention~\ref{item:BK_param}, following our earlier work~\cite{RSW}.

\smallskip 

\textbf{Structure of the rest of Section~\ref{sec:intro}.} In Section~\ref{subsec:rdcp_intro} we rigorously define the RDCP on $K_n$. In Section~\ref{subsec:intro_tc} we introduce the notion of the critical time $t_c$ (using the time parametrization~\ref{item:WW_parametrization}) by recalling some of the main results of \cite{WW}. In Section~\ref{subsec:almost_2reg} we introduce the notation specific to the almost $2$-regular RDCP, state our main result on the asymptotics of $t_c$ (using convention~\ref{item:WW_parametrization}) and discuss the similarity of our result and the Molloy--Reed formula despite the difference between the RDCP and the configuration model.
In Section~\ref{subsec:cont_time_intro} we introduce the continuous-time RDCP (this time parametrization coincides with convention~\ref{item:BK_param} in the fluid limit) and re-state our main result in terms of the asymptotics of the continuous-time critical time $\hat{t}_c$ of the almost $2$-regular RDCP. In Section~\ref{subsec:rel_lit}
we recall some related results from the literature that help to put our main result in context. In Section~\ref{subsec:intro_structure_of_rest} we outline the structure of the rest of this paper.

\subsection{The random degree constrained process}\label{subsec:rdcp_intro}

In Section~\ref{subsec:rdcp_intro} we rigorously introduce the discrete-time random degree constrained process (RDCP) on the complete graph~$K_n$.

Let the vertex set of the graph be $[n]=\{1,\dots,n\}$. We assign degree constraints to each vertex of the graph.
The degree constraints are independent random variables with the same probability distribution. Let ${\underline{p}=(p_k)_{k=2}^{\infty}}$ be a probability distribution on $\{\, 2, 3, \dots \,\}$, where $p_k$ is the probability that a vertex has a maximum allowed degree of~$k$. We assume that there is a universal upper bound $\Delta \in \mathbb{N}_+$ such that 
\begin{equation}\label{eq:upper-bound-assumption}
p_k = 0 \quad \text{for any} \quad k > \Delta.
\end{equation}
Independently of the others, each vertex~$v \in [n]$ is assigned a \textit{degree constraint} $d(v)$ drawn from this distribution~$\underline{p}$, i.e.,~$p_k$ is approximately the fraction of vertices with degree constraint~$k$. Note that by~\eqref{eq:upper-bound-assumption}, each degree constraint is at most~$\Delta$.

\begin{definition}[RDCP in discrete time with host graph $K_n$, {\cite[Definition~1.13]{RSW}}]\label{def:rdcp_disc}
We define the discrete-time random degree constrained process~$\big(G^{n,\ell}_{\underline{p}}\big)_{\ell \ge 0}$, where $\ell \in \mathbb{N}$ is a step parameter and~$G^{n,0}_{\underline{p}}$ is the empty graph with vertex set~$[n]$.
For any~${\ell \in \mathbb{N}_+}$ we then obtain $G^{n,\ell}_{\underline{p}}$ by adding exactly one new edge to $G^{n,\ell-1}_{\underline{p}}$. 
This new edge is chosen uniformly from the set of edges not in~$G^{n,\ell-1}_{\underline{p}}$, for which both end-vertices have degree strictly less than their own degree constraint. If no more edges can be added to the graph without violating any degree constraint, we arrive at the final graph. 
We denote by~$M_n(\underline{p})$ the random number of edges in the final graph $G^{n,M_n(\underline{p})}_{\underline{p}}$. 
For mathematical convenience, we also define $G^{n,\ell}_{\underline{p}}:=G^{n,M_n(\underline{p})}_{\underline{p}}$ for $\ell > M_n(\underline{p})$. 
\end{definition}

It follows e.g.\ from the results of~\cite{RSW} that $M_n(\underline{p})/n$ converges in probability to $\frac{1}{2}\mathbb{E}(D)$ as $n \to \infty$, where~$D$ is a random variable with distribution $\underline{p}$, i.e., $\mathbb{E}(D)=\sum_{k=2}^{\Delta} k\cdot p_k$.

\begin{remark}[On the choice of the host graph]
For convenience, we only defined the RDCP on the complete graph~$K_n$ in Definition~\ref{def:rdcp_disc}. This is the standard choice in the literature (see e.g.~\cite{H24a, MSW, RW92, RW97, RW02, RW23, Se13, WW}). However, it is worth noting that in~\cite{RSW} the RDCP is defined on more general host graphs. There, the authors allow underlying host graphs $G^n$ that form a high-degree almost regular graph sequence in the sense of~\mbox{\cite[Definition~1.3]{NP}} (loosely speaking, a graph sequence $(G^n)$ for which most of the degrees in $G^n$ are close to $r_n$ and the number of edges is close to $ \frac{1}{2} |V(G^n)| \cdot r_n $ where $r_n \to \infty$), which includes $G^n=K_n$ as a special case. However, according to~\cite[Theorem~1.12]{RSW}, the local weak limit of the RDCP does not depend on the choice of the high degree almost regular host graph sequence~$(G^n)$, but only on the distribution~$\underline{p}$ of the degree constraints. In the current paper we will directly work with the local weak limit object since the critical time of the RDCP coincides with the critical time when the extinction/survival phase transition occurs in the multi-type branching process that arises as the local weak limit. 
Let us note that the question of the uniqueness of the giant component is more subtle (cf.\ \cite[Remark~1.46]{RSW}) as the answer can depend on the choice of the sequence of high-degree almost regular host graphs. Uniqueness of the giant is currently only proved when $G^n=K_n$, cf.~\cite{WW}.
This is one more reason why we chose to stick with the host graph $K_n$ in Definition~\ref{def:rdcp_disc}.
\end{remark}

\subsection{Critical time}\label{subsec:intro_tc}

In this paper, we study the so-called \textit{critical time} of the RDCP. Warnke and Wormald~\cite[Theorem~1.6]{WW} proved that if $\underline{p}$ satisfies~\eqref{eq:upper-bound-assumption}, then the discrete-time RDCP on the complete graph~$K_n$ undergoes a \textit{phase transition}:
there exists a critical time~${t_c = t_c(\underline{p})}$ such that:
\begin{itemize}
\item if $s<t_c(\underline{p})$ is fixed and~${n \gg 1}$, then the size of the largest connected component of the discrete-time RDCP graph $G_{\underline{p}}^{n, \lfloor n \cdot s \rfloor }$ is $\mathcal{O}(\ln(n))$ with high probability, but

\item if $s>t_c(\underline{p})$ is fixed and $n \gg 1$, then $G_{\underline{p}}^{n, \lfloor n \cdot s \rfloor }$ has a giant connected component (i.e., a component of size~$\Theta(n)$) with high probability. 
\end{itemize}

The main result of this paper concerns the asymptotics of $t_c(\underline{p})$ when $p_2 \to 1$.

\subsection{Almost 2-regular RDCP}\label{subsec:almost_2reg}

We study the critical time of the RDCP in a regime proposed by Warnke and Wormald (see \cite[Section~6]{WW}): the \textit{almost $2$-regular} case. Warnke and Wormald asked about the behavior of the largest component in the final graph in the almost $2$-regular RDCP. We address a different (albeit related) question and study how the critical time behaves as the distribution of degree constraints gets closer to the $2$-regular case.

In our setting, the degree constraints are almost always equal to $2$. Specifically, the probability that a vertex has a degree constraint equal to $2$ is ${p_2 = 1 - \varepsilon}$, where ${\varepsilon>0}$ is very small. 

Let us introduce the notation $\underline{p}=\chi_2$ for the degree constraint distribution $\underline{p}$ corresponding to the $2$-regular case, i.e., when $p_2=1$ and $p_k=0$ for any $k=3, \, \dots, \, \Delta$.

\begin{remark}[The critical time of the random $2$-process]\label{rem:chi_2_discr}
If~${\varepsilon = 0}$ (i.e., if $\underline{p}=\chi_2$), then the degree constraint of every vertex is equal to $2$. The number $M_n(\chi_2)$ of edges in the final graph satisfies $\mathbb{P}(M_n(\chi_2)=n)\to 1$ as $n \to \infty$ (cf.~\cite{RW92}); thus, $M_n(\chi_2)/n$ converges to~$1$ in probability as~${n \to \infty}$. Furthermore, the local weak limit of the final graphs is a bi-infinite path (cf.~\cite[Corollary~1.18]{RSW}), and the final graph has macroscopic cycles (cf.~\cite[Theorem~1.3]{WW}). However, for any ${0\leq s<1}$, the largest connected component of~$G_{\chi_2}^{n, \lfloor n \cdot s \rfloor }$ is of size $\mathcal{O}(\ln(n))$ with high probability (cf.~\cite[Remark~1.4]{WW}). Thus, in this case we have~${t_c(\chi_2)=1}$.
\end{remark}

We answer the natural question of how fast the critical time~$t_c(\underline{p})$ converges to $1$ in the almost $2$-regular setting as ${\varepsilon \to 0_+}$. Our main result provides a precise asymptotic formula for this relationship, see Theorem~\ref{thm:tc_asymp_behavior_disc}. Before stating this result, we introduce some notation.

\begin{notation}[Space of degree constraint distributions]\label{not:deg_constraint_distr}
Let $\Delta \in \mathbb{N}_+$. Let us denote by $\Pspace$ the set of probability measures on $\{\, 2, 3, \dots, \Delta \, \}$, i.e.,
\begin{equation*}
\Pspace := \left\{ \,(p_2, \, p_3, \, \dots, \, p_\Delta ) \; \Big| \; p_k \ge 0, \; k=2, \, 3, \, \dots, \, \Delta, \; \sum \limits_{d=2}^\Delta p_d = 1 \, \right\}.
\end{equation*}

For any $\underline{p} \in \Pspace$, we define $p_k := 0$ for $k=0$, $k=1$ and for each $k > \Delta$ (cf.~\eqref{eq:upper-bound-assumption}). 

Recall that $\chi_2 \in \Pspace$ is the degree constraint distribution corresponding to the $2$-regular case, i.e., when $p_k = \indic{k=2}$. We also introduce the set
\begin{equation}\label{eq:non_reg}
\nonreg := \Pspace \setminus \chi_2 = \left\{ \,(p_2, \, p_3, \, \dots, \, p_\Delta ) \; \Big| \; 0 \le p_2 < 1, \; p_k \ge 0, \; k=3, \, 4, \, \dots, \, \Delta, \; \sum \limits_{d=2}^\Delta p_d = 1 \, \right\}. 
\end{equation}
\end{notation}

\begin{notation}[Convergence in distribution]
Let $(\underline{p}^m)_{m=1}^\infty$ be a sequence in the space $\Pspace$. We write $\underline{p}^m \Rightarrow \underline{p}$ if $\underline{p}^m$ converges in distribution to~$\underline{p}$ as~$m \to \infty$, i.e., if~$p^m_k \to p_k$ as $m \to \infty$ for any $k \in \{2, \, \dots, \, \Delta\}$. In particular, ${\underline{p}^m \Rightarrow \chi_2}$ is equivalent to~${p^m_2 \to 1}$.
\end{notation}

Now we are ready to state our main theorem. Recall the definition of~$t_c(\underline{p})$ from Section~\ref{subsec:intro_tc}.

\begin{theorem}[Almost $2$-regular asymptotics of the critical time]\label{thm:tc_asymp_behavior_disc}
Let $(\underline{p}^m)_{m=1}^\infty$ be a sequence in $\nonreg$ such that $\underline{p}^m \Rightarrow \chi_2$ as $m \to \infty$. Then we have
\begin{equation}\label{eq:tc_asymp_disc}
\lim \limits_{m \to \infty} \frac{1-t_c(\underline{p}^m)}{\sum \limits_{k=3}^\Delta (k^2-3k+2)p^m_k}=\frac 12.
\end{equation}
\end{theorem}

Intuitively, the theorem states that if~$\underline{p}$ is close to the purely $2$-regular distribution~$\chi_2$, then we have
\begin{equation}\label{eq:intuitive_disc}
t_c(\underline{p}) \stackrel{\eqref{eq:tc_asymp_disc}}{\approx}
1- \frac 12 \sum \limits_{k=3}^\Delta (k^2-3k+2)p_k=
1- \frac 12 \sum \limits_{k=1}^\Delta (k^2-3k+2)p_k.
\end{equation}

This formula shows that the critical time~$t_c(\underline{p})$ is slightly less than $t_c(\chi_2) = 1$, reflecting how vertices with degree constraints greater than 2 accelerate the emergence of the giant component compared to the purely $2$-regular case (i.e., slightly fewer edges are required for the giant to appear). In particular, if $p_2=1-\varepsilon$ for sufficiently small $\varepsilon>0$, then the reduction in the critical time~$t_c(\underline{p})$ has the same order of magnitude as $\varepsilon$.

We will deduce Theorem~\ref{thm:tc_asymp_behavior_disc} in Section~\ref{subsec:discrete_from_cont} from the critical time asymptotics of the almost $2$-regular, continuous-time RDCP that we will introduce in Section~\ref{subsec:cont_time_intro}.

\begin{remark}[Heuristic argument for the almost $2$-regular critical time asymptotics using Molloy--Reed]\label{rem:molloy_reed}
Let us provide an admittedly naive argument for \eqref{eq:intuitive_disc} by replacing the RDCP dynamics with the dynamics of the \emph{configuration model} (see e.g.~\cite[Section~1.3.3]{H24b}) with the same degree distribution~$\underline{p}$, and study the emergence of the giant in this simpler model.
For each vertex~${i \in [n]}$, we assign $d_i$ half-edges, where $d_1,\dots,d_n$ are i.i.d.\ random variables with distribution~$\underline{p}$. We then generate a uniform matching of these half-edges to form edges and call the resulting (multi)graph the configuration model.
One way of sequentially generating this graph is to add edges one by one, uniformly choosing a pair from the set of remaining half-edges and connecting them with an edge. Let $H^{n,\ell}_{\underline{p}}$ denote the graph that we obtain by adding the first $\ell$ edges. Let us again denote the total number of edges by~$M_n(\underline{p})$ and note that $M_n(\underline{p})/n \approx \frac{1}{2}\mathbb{E}(D)$, where~${D \sim \underline{p}}$.
It is known that $H^{n,\ell}_{\underline{p}}$ has the same distribution as the subgraph of the final configuration model $H^{n,M_n(\underline{p})}_{\underline{p}}$ that we obtain by picking a uniformly distributed $\ell$-element subset of the edges. Thus, heuristically, $H^{n,\lfloor n\cdot t\rfloor}_{\underline{p}}$ has approximately the same distribution as the subgraph of $H^{n,M_n(\underline{p})}_{\underline{p}}$ that we obtain by keeping each edge independently with probability~$Q$, where
\begin{equation}\label{eq:heur_prob}
Q := \frac{\lfloor n \cdot t \rfloor}{M_n(\underline{p})} \approx \frac{2t}{\ev(D)}.
\end{equation} 

Now we study the critical time $t^*_c(\underline{p})$ of this percolated configuration model. Using the \emph{Molloy--Reed criterion}~\cite{MR}, we can derive that $H^{n,\lfloor n\cdot t\rfloor}_{\underline{p}}$ has a giant component for large $n$ if and only if
\begin{equation}\label{eq:molloy-reed}
\frac{\ev[D(D-1)]}{\ev(D)} \cdot Q > 1,
\quad \text{i.e., by the choice of $Q$ (see~\eqref{eq:heur_prob}),} \quad
t>\frac{(\ev(D))^2}{2\, \ev[D(D-1)]}.
\end{equation}

If~$\underline{p}$ is close to~$\chi_2$, then we can write $D = 2+\xi$, where $\xi$ is a `small' non-negative integer-valued random variable, i.e., $\mathbb{P}(\xi>0) = 1-p_2 =\varepsilon$. 
 Therefore, we obtain
\begin{equation}\label{eq:tc_heur_appr}
t^*_c(\underline{p}) \stackrel{\eqref{eq:molloy-reed}}{\approx} \frac{(\ev(D))^2}{2 \, \ev[D(D-1)]}
= \frac{1+ \ev(\xi)+\frac{1}{4}(\ev(\xi))^2}{1+\frac{3}{2}\ev(\xi)+\frac{1}{2}\ev(\xi^2)} \approx \frac{1+\ev(\xi)}{1+\frac{3}{2}\,\ev(\xi)+\frac{1}{2}\ev(\xi^2)}
\approx 1-\frac 12 \ev(\xi)- \frac 12 \ev(\xi^2).
\end{equation}

Observe that the approximation in~\eqref{eq:tc_heur_appr} matches the result of Theorem~\ref{thm:tc_asymp_behavior_disc}, since
\begin{equation*}
t_c(\underline{p}) \stackrel{\eqref{eq:intuitive_disc}}{\approx} 1- \frac 12 \ev[D^2-3D+2] = 1-\frac 12 \ev [\xi + \xi^2] = 1-\frac 12 \ev(\xi) - \frac 12 \ev(\xi^2). 
\end{equation*}
\end{remark}

\begin{remark}[The flaws of the above heuristic argument]\label{rem:flaws}
Our rigorous proof of Theorem~\ref{thm:tc_asymp_behavior_disc} will not look like the heuristic argument sketched in Remark~\ref{rem:molloy_reed}.
Actually, $t^*_c(\underline{p})$ has no reason to be equal to $t_c(\underline{p})$, because they are the critical times of two genuinely different random graph models:
\begin{itemize}
\item In~\cite{MSW} it is shown that the final graph of the RDCP is not contiguous to the configuration model with the same degree sequence.
\item Let us also note that if we fix the final graph~$H^{n,M_n(\underline{p})}_{\underline{p}}$ (i.e., the final configuration model), then the edge arrival times become asymptotically independent as~${n \to \infty}$: we used this in~\eqref{eq:molloy-reed}.
In contrast, even in the simplest ${\underline{p}=\chi_2}$ case of the RDCP (i.e., the random $2$-process), the deposition times of edges of~$G^{n,M_n(\underline{p})}_{\underline{p}}$ that share an end-vertex are correlated: it follows from the $\underline{p}=\chi_2$ case of \cite[Theorem~1.41]{RSW} that the edge arrival times of the local weak limit object form a hidden Markov chain. 
\item Also, if we denote by~$\chi_3$ the degree distribution concentrated on degree~$3$ (i.e., if we compare the random $3$-regular graph and the random $3$-process), then $t^*_c(\chi_3)=\frac{3}{4}$ by~\eqref{eq:molloy-reed}, while numerically one obtains ${t_c(\chi_3) \approx 0.577}$ using the analytic method of~\cite{WW} as well as the analytic method of~\cite{RSW}. 
\end{itemize}
\end{remark}

\begin{remark}[Our result hints at almost $2$-regular universality]
Despite the reasons listed in Remark~\ref{rem:flaws}, Theorem~\ref{thm:tc_asymp_behavior_disc} does hold, which suggests that the macroscopic structures of the late phase of the almost \mbox{$2$-regular} RDCP and Bernoulli bond percolation on the almost $2$-regular configuration model with parameter~${Q \approx 1}$ look similar.
Note that the results of~\cite{DKLP11} give a precise description of the structure of an almost $2$-regular random graph, namely the $2$-core of the largest connected component of a barely supercritical Erd\H{o}s--R\'enyi graph: loosely speaking, it is a configuration model in which we replace edges with long paths of i.i.d.\ geometrically distributed length.
A variant of this description works throughout the supercritical phase of the Erd\H{o}s--R\'enyi graph, cf.~\cite{DLP14}.
The final graph of the almost $2$-regular RDCP and the almost $2$-regular configuration model might have a similar description. 
\end{remark}

\subsection{Continuous-time RDCP}\label{subsec:cont_time_intro}

The RDCP is constructed by attempting to add the edges of $K_n$ one-by-one and an edge is added only if it does not violate the degree constraints. There are different ways to parametrize time: loosely speaking, Definition~\ref{def:rdcp_disc}
 measures time by the number of successfully deposited edges,
 while e.g.\ in~\cite{BK} time is measured by the number of edge deposition attempts.
In our proof we will use a time parametrization similar to that of~\cite{BK}. More specifically, we use a natural continuous-time version of the RDCP, defined in~\cite[Definition~1.4]{RSW}. 

\begin{definition}[Continuous-time RDCP on $K_n$, {\cite[Definition~1.4]{RSW}}]
The random degree constrained process $\big(G^n_{\underline{p}}(\hat{t})\big)_{\hat{t} \geq 0}$ is a time-evolving random graph process on the complete graph~$K_n$, where~$\hat{t}$ is a \textit{continuous time} parameter and $G^n_{\underline{p}}(0)$ is the empty graph with vertex set~$[n]$. For each edge $e \in E(K_n)$, we independently let $X_e \sim \text{EXP}\left(1/n\right)$ be the \textit{activation time} of~$e$. As time evolves, we attempt to add edges to the graph one-by-one, in increasing order of their activation times. An edge~${e=vw}$ is added at its activation time $X_e$ if and only if the degrees of vertices $v$ and~$w$ in the RDCP graph~$G^n_{\underline{p}}(X_e^-)$ right before time~$X_e$ are strictly less than their degree constraints~$d(v)$ and~$d(w)$, respectively. If no more edges can be added to the graph without violating any degree constraint, we arrive at the \textit{final graph} of the RDCP on the complete graph~$K_n$, which we denote by~$G^n_{\underline{p}}(\infty)$.
\end{definition}

Recall the definition of the critical time $t_c(\underline{p})$ of $\big(G^{n,\ell}_{\underline{p}}\big)_{\ell \ge 0}$ from Section~\ref{subsec:intro_tc}.
The following definition is actually a corollary of~\cite[Theorem~1.6]{WW} and \cite[Proposition~1.20]{RSW} due to the time-change result stated in~\cite[Corollary~1.18]{RSW}. 

\begin{definition}[Critical time of the continuous-time RDCP on $K_n$]\label{def:t_crit_cont}
We define the critical time~${\hat{t}_c = \hat{t}_c(\underline{p})}$ of $\big(G^n_{\underline{p}}(\hat{t})\big)_{\hat{t} \geq 0}$ as follows:
if $\hat{t}<\hat{t}_c(\underline{p})$ is fixed and~${n \gg 1}$, then the size of the largest connected component of~$G^n_{\underline{p}}(\hat{t})$ is $\mathcal{O}(\ln(n))$ with high probability, but if $\hat{t}>\hat{t}_c(\underline{p})$ is fixed and $n \gg 1$, then $G^n_{\underline{p}}(\hat{t})$ has a giant connected component with high probability. 
\end{definition}

\begin{remark}[The critical time of the continuous-time random $2$-process]\label{rem:2reg_hattc_infty}
Note that $\hat{t}_c(\chi_2) = \infty$, because if $0\leq \hat{t} <+\infty$, then (loosely speaking) 
$G^n_{\chi_2}(\hat{t})$ corresponds to $G^{n,\lfloor n \cdot s \rfloor}_{\chi_2}$ for some $0 \leq s < \frac{1}{2}\mathbb{E}(D)= 1$ under the time-change stated in \cite[Corollary~1.18]{RSW}, and we have already discussed that $G^{n,\lfloor n \cdot s \rfloor}_{\chi_2}$ is subcritical for any $0 \leq s <1$ in Remark \ref{rem:chi_2_discr}.
\end{remark}

Thus the question that we address is: how fast does $\hat{t}_c(\underline{p})$ converge to $\infty$ if $\underline{p}$ weakly converges to $\chi_2$?

\begin{theorem}[Asymptotics of the critical time of the almost $2$-regular continuous-time RDCP]\label{thm:tc_asymp_behavior_cont}
Let $(\underline{p}^m)_{m=1}^\infty$ be a sequence in $\nonreg$ such that $\underline{p}^m \Rightarrow \chi_2$ as $m \to \infty$. Then we have
\begin{equation}\label{eq:tc_asymp_cont}
\lim \limits_{m \to \infty} \hat{t}_c(\underline{p}^m) \cdot \sum \limits_{k=3}^\Delta k \cdot (k-2) \cdot p^m_k =1.
\end{equation}
\end{theorem}

Intuitively,~\eqref{eq:tc_asymp_cont} states that
\begin{equation*}
\tc \approx \frac{1}{\sum \limits_{k=3}^\Delta k(k-2)p_k} = \frac{1}{\sum \limits_{k=2}^\Delta k(k-2)p_k} = \frac{1}{\ev[D(D-2)]},
\end{equation*}
where~$D$ is a random variable with distribution~$\underline{p}$, and $\underline{p} \in \nonreg$ is almost $2$-regular. Loosely speaking, if $1-p_2=\varepsilon$ and $\varepsilon \ll 1$, then the order of magnitude of $\tc$ is ${1}/{\varepsilon}$.

\subsection{Further related literature}\label{subsec:rel_lit}

Apart from the papers mentioned earlier, let us mention some more.

\smallskip

\textbf{Random $\bm{d}$-process.}
The random $d$-process has been studied for decades, even by chemists and physicists~\cite{BQ87, KQ}. 
Over the years many properties of the random \mbox{$d$-process} have been studied, e.g.\ the fluid limit for the time evolution of the degree distribution \cite{W95,W99a}, the connectedness~\cite{RW02}, the distribution of short cycles~\cite{RW97}, Hamiltonicity~\cite{TWZ}, the time evolution of the minimal degree~\cite{H24a} and the exact asymptotics of the probability of non-saturation~\cite{RW23}. 
In the paper~\cite{Se13} a central limit theorem is proved for the size of the giant component in the random $d$-process, if it exists: the author does not prove that these graph processes actually have a giant component; he merely shows that if there is a giant component, then the number of vertices it contains is asymptotically normally distributed.

\smallskip

\textbf{Barely supercritical random graphs.} 
The barely supercritical regime of the configuration model (e.g.\ uniqueness of the giant component) is studied in~\cite{HJL19}. The paper~\cite{RW10} provides precise results on the diameter of the Erd\H{o}s--R\'enyi graph in various regimes, including the barely supercritical one.

\smallskip 

\textbf{Perron--Frobenius theory and the phase transition of random graphs.}
The papers~\cite{BJR07} and~\cite{GM15} use multi-type branching processes as a local proxy of inhomogeneous random graphs and consider the principal eigenvalue~$\mu$ of the associated branching operator, just like us: loosely speaking, the graph has a giant connected component if and only if~${\mu>1}$.

\smallskip

\textbf{Dynamic random graphs and Flory-type coagulation equations.}
We have already mentioned that Flory's coagulation equations describe the fluid limit of the component-size distribution of the dynamical Erd\H{o}s--R\'enyi graph process. Similarly, Smoluchowski's coagulation equations \cite{S16} describe self-organized critical (SOC) variants of the Erd\H{o}s--R\'enyi graph process in which incipient giants are immediately removed \cite{MN14, R09}. The $p$-frozen Erd\H{o}s--R\'enyi model \cite{HV25} is similarly related to a modification of Smoluchowski's coagulation equations \cite{K24}.
Smoluchowski’s equation with limited aggregations \cite{B09} also has a corresponding dynamic random graph model \cite{MN15}, which is a SOC modification of the dynamic configuration model. The mean field forest fire model (FFM)~\cite{RT09} is also a SOC modification of the Erd\H{o}s--R\'enyi graph process, which also has its own Smoluchowski-type differential equations, but the structure of the FFM turns out to be closely related to the inhomogeneous random graph model of~\cite{BJR07}. Thus, an alternative approach to the FFM which involves multi-type branching processes and Perron--Frobenius theory is developed in~\cite{CRY21}.

\subsection{Structure of the rest of this paper}
\label{subsec:intro_structure_of_rest}

In Section~\ref{subsec:discrete_from_cont} we derive Theorem~\ref{thm:tc_asymp_behavior_disc} (which uses the discrete time parametrization) from Theorem~\ref{thm:tc_asymp_behavior_cont} (which uses the continuous time parametrization).
We prove Theorem~\ref{thm:tc_asymp_behavior_cont} in Sections~\ref{subsec:cont_time_proof} and~\ref{subsec:eigenfunction_bounds}. Similarly to the proof of \cite[Theorem~1.22]{RSW}, our result relies on a spectral characterization of the critical time $\tc$ combined with classical results of Perron--Frobenius theory. In Section~\ref{subsec:cont_time_proof} we transform the problem of giving upper and lower bounds on the critical time $\tc$ into the problem of constructing appropriate test functions that certify that the principal eigenvalue of the branching operator shortly before $\tc$ is less than one, but shortly after $\tc$ it is greater than one.
In Section~\ref{subsec:eigenfunction_bounds} we construct the above-mentioned test functions via the perturbative analysis of a second-order linear ODE that characterizes the principal eigenfunction.

\section{Asymptotic behavior of the critical time}

\subsection{Connection between discrete- and continuous time asymptotics}\label{subsec:discrete_from_cont}

In this section we derive Theorem~\ref{thm:tc_asymp_behavior_disc} (which uses the discrete time parametrization) from Theorem~\ref{thm:tc_asymp_behavior_cont} (which uses the continuous time parametrization). For that purpose, we need to introduce some notation that facilitates this time-change, following~\cite{RSW}. The function~$\lamp$ (cf.\ Notation~\ref{not:lambda}) will play an important role in the rest of the paper. 

\begin{notation}[Tail probability]
Let $\underline{p} \in \Pspace$. We define $q_k$ as in \cite[Notation~1.34]{RSW}, i.e.,
\begin{equation}\label{eq:qk_def}
q_k := \sum \limits_{d=k}^\Delta p_d
\quad \text{for} \quad
k=1,\, 2,\, \dots, \, \Delta
\qquad \text{and} \qquad
q_k :=0
\quad \text{for} \quad
k > \Delta.
\end{equation}
Observe that for any $\underline{p} \in \Pspace$, we have $q_1 = q_2 = 1$.
\end{notation}

Now we recall the definition of the function~$\lamp$ from \cite{RSW}, the probabilistic meaning of which is explained in 
\cite[Remarks~1.30 and~1.35]{RSW}.

\begin{notation}[Function~$\lamp$]\label{not:lambda}
Let $\underline{p} \in \nonreg$. We introduce the notation $\lamp(t)$ following \cite[Notation~1.34]{RSW}, but in this paper we emphasize its dependence on the distribution~$\underline{p}$ with the subscript in order to distinguish it from the $2$-regular case. Namely,
$\lamp(t)$ is the solution of the initial value problem
\begin{equation}\label{eq:lambda_p_ode_with_q}
\lambda'_{\underline{p}}(t) = \mathrm{e}^{-\lamp(t)} \cdot \sum \limits_{k=0}^{\Delta-1} \frac{\lamp^k(t)}{k!} q_{k+1}, \qquad \lamp(0) = 0.
\end{equation}

In the $2$-regular case, we omit the subscript, i.e., $\lambda(t) := \lambda_{\chi_2}(t)$ is the solution of the initial value problem
\begin{equation}\label{eq:lambda_2reg_ode}
\lambda'(t) = \mathrm{e}^{-\lambda(t)} \cdot (1+\lambda(t)), \qquad \lambda(0) = 0.
\end{equation}
\end{notation}

Recall from \cite[Section 1.3]{RSW} that~$G^\infty_{\underline{p}}(\hat{t})$ denotes the local limit of the continuous-time RDCP at time $\hat{t}$.

\begin{definition}[Mean number of neighbors of the root, {\cite[Definition~1.16]{RSW}}]\label{def:F_p}
Given $\hat{t} \in \mathbb{R}_+ \cup \{ \infty\}$, let $F_{\underline{p}}(\hat{t})$ denote the expected number of neighbors of the root in the local limit $G^\infty_{\underline{p}}(\hat{t})$.
\end{definition}

The next proposition summarizes the results of \cite[Proposition~1.20, Lemma~4.10, Claim~4.12]{RSW}.

\begin{proposition}[Properties of $F_{\underline{p}}$]
Consider the function $F_{\underline{p}}$ defined in Definition~\ref{def:F_p} and let $D$ be a random variable with distribution~$\underline{p}$. We have
\begin{align}
\lim \limits_{\hat{t} \to \infty} F_{\underline{p}}(\hat{t})&=\ev (D)=\sum_{k=2}^{\Delta} k\cdot p_k, \label{eq:Fp_limit}\\
F_{\underline{p}}(\hat{t})&= \int \limits_0^{\hat{t}} \left( \lambda'_{\underline{p}}(s) \right)^2 \, \mathrm{d}s
\quad \text{for any} \quad \hat{t} \in \mathbb{R}_+, \label{eq:Fp_integral_formula}
\end{align}
where $\lamp(t)$ is the solution of the initial value problem~\eqref{eq:lambda_p_ode_with_q}.
Moreover, we have
\begin{equation}\label{eq:tc_connection}
t_c(\underline{p}) = \frac 12 F_{\underline{p}}\left(\hat{t}_c(\underline{p})\right),
\end{equation}
where~$t_c(\underline{p})$ is defined in Section~\ref{subsec:intro_tc} and~$\tc$ is defined in Definition~\ref{def:t_crit_cont}.
\end{proposition}

In order to state and prove some of our asymptotic results, we will use the following notation.

\begin{notation}[Asymptotic relations]\label{not:asymp_relations}
We define the notation $h(t) \sim g(t)$ and $h(t) \ll g(t)$ for functions $h(t)\colon \mathbb{R}_+ \to \mathbb{R}$, $g(t)\colon \mathbb{R}_+ \to \mathbb{R}$ in the following way:
\begin{alignat*}{2}
&h(t) \sim g(t) &&\quad \text{if and only if} \quad \lim \limits_{t \to \infty} \frac{h(t)}{g(t)} = 1,\\
&h(t) \ll g(t) &&\quad \text{if and only if} \quad \lim \limits_{t \to \infty} \frac{h(t)}{g(t)} = 0.
\end{alignat*}
\end{notation}

We will use the following claim to derive Theorem~\ref{thm:tc_asymp_behavior_disc} from Theorem~\ref{thm:tc_asymp_behavior_cont}.

\begin{claim}[Asymptotics of an integral]
The following asymptotic relation holds uniformly in~$\underline{p} \in \Pspace$:
\begin{equation}\label{eq:int_lam_p_diff^2_asymp}
\int \limits_t^\infty \left(\lamp'(s)\right)^2 \, \mathrm{d}s \sim \frac 1t.
\end{equation}
\end{claim}

\begin{proof}
Let $\underline{p} \in \Pspace$ and recall the corresponding~$(q_k)_{k=1}^{\Delta}$ from~\eqref{eq:qk_def}. Let us introduce the functions
\begin{equation}\label{eq:P_gamma_psi_def}
P_{\underline{p}}(\lambda) := \sum \limits_{k=0}^{\Delta-1} \frac{\lambda^k}{k!} q_{k+1},
\qquad
\gamma_{\underline{p}}(\lambda) := \frac{\mathrm{e}^{\lambda}}{P_{\underline{p}}(\lambda)}
\qquad \text{and} \qquad
\psi_{\underline{p}}(\lambda) := \int \limits_0^\lambda \gamma_{\underline{p}}(s) \, \mathrm{d}s,
\qquad \lambda \in \mathbb{R}_+.
\end{equation}

Note that we have
$\lamp'(t) = 1/\gamma_{\underline{p}}(\lamp(t))$ (cf.~\eqref{eq:lambda_p_ode_with_q}), therefore, $t = \psi_{\underline{p}}(\lamp(t))$ and $\lamp(t) = \psi_{\underline{p}}^{-1}(t)$. Moreover, for any $\underline{p} \in \Pspace$ and for any $\lambda \in \mathbb{R}_+$, we have $\gamma_{\underline{p}}(\lambda) \le \exp(\lambda)$, and therefore, $\psi_{\underline{p}}(\lambda) \le \exp(\lambda)-1$. Hence, $\lamp(t) \ge \ln (t+1)$ for any $t \in \mathbb{R}_+$.

In order to prove~\eqref{eq:int_lam_p_diff^2_asymp}, it is enough to show
\begin{equation}\label{eq:goal_with_gamma}
\lamp'(t) \sim \frac 1t,
\qquad \text{or equivalently,} \qquad
\lim \limits_{t \to \infty} \frac{\gamma_{\underline{p}}(\lamp(t))}{\psi_{\underline{p}}(\lamp(t))} = 1.
\end{equation}

Using l'H\^{o}pital's rule and~\eqref{eq:P_gamma_psi_def}, we obtain that~\eqref{eq:goal_with_gamma} is equivalent to
\begin{equation}\label{eq:goal_with_P}
\lim \limits_{t \to \infty} \frac{P_{\underline{p}}'(\lamp(t))}{P_{\underline{p}}(\lamp(t))} = 0.
\end{equation}

Thus, our goal is to show that \eqref{eq:goal_with_P} holds uniformly in~${\underline{p} \in \Pspace}$.
For any $\underline{p} \in \Pspace$, we have
\begin{equation}\label{eq:dP/P_uniform_bound}
0 \le P_{\underline{p}}'(\lambda)
\stackrel{\eqref{eq:P_gamma_psi_def}}{\le}
\frac{\Delta-1}{\lambda} \cdot P_{\underline{p}}(\lambda),
 \quad
\lambda \in \mathbb{R}_+,
\qquad \text{i.e.,} \qquad
0 \le \frac{P_{\underline{p}}'(\lamp(t))}{P_{\underline{p}}'(\lamp(t))} \le \frac{\Delta-1}{\ln (t+1)},
 \quad t \in \mathbb{R}_+.
\end{equation}

By~\eqref{eq:dP/P_uniform_bound}, we obtain that~\eqref{eq:goal_with_P} holds uniformly in~${\underline{p} \in \Pspace}$. Therefore,~\eqref{eq:goal_with_gamma} and~\eqref{eq:int_lam_p_diff^2_asymp} also hold uniformly in~${\underline{p} \in \Pspace}$.
\end{proof}

\begin{proof}[Proof of Theorem~\ref{thm:tc_asymp_behavior_disc} (using Theorem~\ref{thm:tc_asymp_behavior_cont})]
Let $(\underline{p}^m)_{m=1}^\infty$ denote a sequence in $\nonreg$ such that $\underline{p}^m \Rightarrow \chi_2$ as~${m \to \infty}$ and let~$D^m$ be a random variable with distribution~$\underline{p}^m$. In this proof we use the notation $h_m \sim g_m$ for sequences $h_m\colon \mathbb{N}_+ \to \mathbb{R}$, $g_m\colon \mathbb{N}_+ \to \mathbb{R}$ if ${\lim h_m/g_m = 1}$ as~${m \to \infty}$. We also use $h_m=o(g_m)$ if
${\lim h_m/g_m = 0}$ as~${m \to \infty}$. Let us denote $\varepsilon^m:=1-p_2^m=\sum_{k=3}^{\Delta} p^m_k$.

Note that $\hat{t}_c(\underline{p}^m) \to \infty$ as $m \to \infty$ (cf.~\eqref{eq:tc_asymp_cont}) and
$\int_{\hat{t}_c(\underline{p}^m)}^\infty \left( \lambda_{\underline{p}^m}'(s) \right)^2 \, \mathrm{d}s \stackrel{\eqref{eq:int_lam_p_diff^2_asymp}}{\sim}
\frac{1}{\hat{t}_c(\underline{p}^m)} \stackrel{\eqref{eq:tc_asymp_cont}}{\sim} \ev[D^m(D^m-2)]$. 

This implies the relation $\int_{\hat{t}_c(\underline{p}^m)}^\infty \left( \lambda_{\underline{p}^m}'(s) \right)^2 \, \mathrm{d}s = \ev[D^m(D^m-2)] + o(\varepsilon^m)$ that we use in $(*)$ below:
\begin{align}
\begin{split}\label{eq:tc_with_expectation}
1-t_c(\underline{p}^m) \stackrel{\eqref{eq:tc_connection}}{=}&
1-\frac 12 F_{\underline{p}}(\hat{t}_c(\underline{p}^m))
\stackrel{\eqref{eq:Fp_integral_formula}, \, \eqref{eq:Fp_limit}}{=}
1-\frac 12 \left[ \ev(D^m) - \int \limits_{\hat{t}_c(\underline{p}^m)}^\infty \left( \lambda_{\underline{p}^m}'(s) \right)^2 \, \mathrm{d}s \right]
\stackrel{(*)}{=}\\
&1-\frac 12 \left[ \ev(D^m) - \ev(D^m(D^m-2)) + o(\varepsilon^m) \right] = \frac 12 \ev[(D^m)^2-3D^m+2]+o(\varepsilon^m).
\end{split}
\end{align}
Equation~\eqref{eq:tc_with_expectation} implies Theorem~\ref{thm:tc_asymp_behavior_disc}.
\end{proof}

\subsection{Asymptotic behavior in the continuous-time RDCP}\label{subsec:cont_time_proof}

In this section we state the key Lemma~\ref{lem:eigenfunction_bounds} and derive our main result Theorem~\ref{thm:tc_asymp_behavior_cont} from it.
For that purpose, we introduce~$\tunder$ and~$\tover$ that will serve as lower and upper bounds for the critical time~$\tc$ (the smaller the~$\delta$, the closer these bounds are).
We will introduce an operator $\Toperator{\hat{t}}$ (closely related to the branching operator of the multi-type branching process that arises as the local weak limit $G^\infty_{\underline{p}}(\hat{t})$ of the RDCP graph $G^n_{\underline{p}}(\hat{t})$ at time $\hat{t}$). In Theorem~\ref{thm:T_operator_prop} below we use the results of \cite{RSW} to prove (among other things) that if $\hat{t} \leq \tc$ then the principal eigenvalue of $\Toperator{\hat{t}}$ is less than or equal to $1$, but if $\hat{t}\geq \tc$ then it is greater than or equal to $1$.
Lemma~\ref{lem:eigenfunction_bounds} states that if we fix some small but positive $\delta$, then for $\underline{p}$ close enough to $\chi_2$, we can construct non-negative functions~$\wunder$ and~$\wover$ such that $\Toperator{\tunder}\wunder \le \wunder$ and $\Toperator{\tover}\wover \ge \wover$. We will then show that this implies $\tunder \leq \tc \leq \tover$, and conclude Theorem~\ref{thm:tc_asymp_behavior_cont}. 
 
Recall Notation~\ref{not:deg_constraint_distr}.

\begin{definition}[Approximations of the critical time]
Let $\underline{p} \in \nonreg$ and $\delta \in (0,1]$. Let us define
\begin{equation}\label{eq:tc_approx}
\tunder := \frac{1-\delta}{\sum \limits_{k=3}^\Delta k(k-2)p_k},
\qquad\qquad
\tover := \frac{1+\delta}{\sum \limits_{k=3}^\Delta k(k-2)p_k}.
\end{equation}
\end{definition}

We will prove Theorem~\ref{thm:tc_asymp_behavior_cont} by showing that if $\underline{p}$ is close enough to the $2$-regular distribution~$\chi_2$, then the values defined in~\eqref{eq:tc_approx} squeeze the critical time~$\tc$, i.e., ${\tunder \le \tc \le \tover}$. Before that, we need to introduce the probability density function $\Hp$.

\begin{notation}[The density function~$H_{\underline{p}}$]\label{not:Hp}
Let~$\underline{p} \in \nonreg$. We introduce the notation~$\Hp(t)$ following \cite[Definition~4.7, Lemma~4.8]{RSW}. Namely, $\Hp(t)$ is the probability density function of $\tau_{\underline{p}}$, where
\begin{equation}\label{eq:tau_def}
\tau_{\underline{p}} := \min \{\, t \,:\, X_{\underline{p}}(t) = D_{\underline{p}}-1 \,\},
\end{equation}
where $X_{\underline{p}}(t)$ is the counting process of an inhomogeneous Poisson point process on $\mathbb{R}_+$ with intensity function~$\lambda'_{\underline{p}}(t)$ (in particular, $X_{\underline{p}}(t) \sim \text{POI}(\lamp(t))$), and $D_{\underline{p}}$ denotes a random variable with distribution $\underline{p}$, which is independent of the process $X_{\underline{p}}(t)$.

Similarly to $\lamp$ (cf.~Notation~\ref{not:lambda}), we emphasize the dependence of $\Hp(t)$ on the distribution~$\underline{p}$ with the subscript, and in the $2$-regular case we omit the subscript, i.e., we let $H(t) := H_{\chi_2}(t)$.
\end{notation}

In \cite[Theorem~1.45]{RSW} a spectral characterization of the critical time $\tc$ is given in terms of the so-called branching operator $B_{\underline{p}, \hat{t}}$ (that was defined in \cite[Definition~1.43]{RSW}) of the multi-type branching process that reproduces the local weak limit $G^\infty_{\underline{p}}(\hat{t})$ (see \cite[Theorem~1.41]{RSW}). Instead of the branching operator $B_{\underline{p}, \hat{t}}$, we will work with the operator $\Toperator{\hat{t}}$ (see Definition \ref{def:T_operator} below), but let us 
already note that 
$\Toperator{\hat{t}}$ has a clear connection to $B_{\underline{p}, \hat{t}}$ that we will point out in the proof of Theorem~\ref{thm:T_operator_prop}.

\begin{definition}[Transformed branching operator]\label{def:T_operator}
Let ${\underline{p} \in \Pspace}$, ${\hat{t} \in \mathbb{R}_+ \cup \{+\infty\}}$ and ${\varphi:\mathbb{R}_+ \to \mathbb{R}}$. We define the linear operator ${\Toperator{\hat{t}}\varphi\colon \mathbb{R}_+ \to \mathbb{R}}$ by
\begin{equation}\label{eq:T_operator_def}
(\Toperator{\hat{t}}\varphi)(t) := \int \limits_0^\infty \Hp(s)\cdot \varphi(s) \cdot (\hat{t} \wedge t \wedge s) \, \mathrm{d}s, \qquad t \in \mathbb{R}_+.
\end{equation}
\end{definition}

In \cite[Theorem~1.45, Lemma~4.15]{RSW} it was proved that the branching operator~$B_{\underline{p},\hat{t}}$ has various nice and useful properties. Moreover, a spectral characterization of the critical time~$\tc$ was given in terms of the branching operator~$B_{\underline{p},\hat{t}}$. Now we show that the operator~$\Toperator{\hat{t}}$ inherits these properties.

\begin{theorem}[Properties of the operator $\Toperator{\hat{t}}$]\label{thm:T_operator_prop}
For any $\underline{p} \in \Pspace$ and $\hat{t} \in \mathbb{R}_+ \cup \{ + \infty\}$, the operator $\Toperator{\hat{t}}$ satisfies the following properties.
\begin{enumerate}
\item\label{item:T_operator_sa_HS} The operator $\Toperator{\hat{t}}$ is self-adjoint and Hilbert--Schmidt on $L^2(\mathbb{R}_+,\, \Hp)$, and the $L^2(\mathbb{R}_+,\, \Hp)$ operator norm $\norm{\Toperator{\hat{t}}}$ of $\Toperator{\hat{t}}$ is~finite.

\item\label{item:op_norm_mon_in_t} The function $\hat{t} \mapsto \norm{\Toperator{\hat{t}}}$ is strictly increasing and continuous on $[0,+\infty)$.

\item\label{item:op_norm_1_at_tc} We have $\norm{\Toperator{\tc}}=1$, i.e., the norm of the operator $\Toperator{\hat{t}}$ at the critical time $\hat{t}=\tc$ equals~one.

\item\label{item:princip_eigenvalue} $\norm{\Toperator{\hat{t}}}$ is the principal eigenvalue of $\Toperator{\hat{t}}$ with multiplicity 1 and with positive eigenfunctions.

\item The only eigenvalue with a non-negative eigenfunction is $\norm{\Toperator{\hat{t}}}$.
\end{enumerate}
\end{theorem}

\begin{proof}
We use the notation $z_{\underline{p},k}^t$, $f_{\underline{p}}(t)$, $E_{\underline{p}}(t)$, $\rho_{\underline{p}}(t)$ for the objects defined in \cite[Notation 1.34, Claim~1.36, Claim~4.3, Definition~4.7]{RSW}, but with a subscript to emphasize their dependence on the distribution $\underline{p}$:
\begin{equation}\label{eq:formulasss}
z_{\underline{p},k}^t := \mathrm{e}^{-\lamp(t)} \cdot \frac{\lamp^k(t)}{k!}\cdot p_{k+1}, \quad
f_{\underline{p}}(t) := \lamp'(t) \cdot \sum \limits_{k=1}^\infty z_{\underline{p},k}^t, \quad E_{\underline{p}}(t) := \frac{\sum \limits_{k=0}^{\infty} k \cdot z_{\underline{p},k}^t}{\sum \limits_{k=0}^{\infty} z_{\underline{p},k}^t}, \quad 
\rho_{\underline{p}}(t):= \frac{\lamp(t)f_{\underline{p}}(t)}{E_{\underline{p}}(t)}.
\end{equation}

It was proved in \cite[Claim~4.4]{RSW} that the branching operator~$B_{\underline{p},\hat{t}}$ (cf.\ \cite[Definition~1.42]{RSW}) associated to the multi-type branching process (cf.\ \cite[Definition~1.39]{RSW}) that arises as the local weak limit of the RDCP (cf.\ \cite[Theorem~1.41]{RSW})
acts on a function ${v \in L^2(\mathbb{R}_+,\, \rho_{\underline{p}})}$ as
\begin{equation}\label{eq:branching_op}
(B_{{\underline{p}}, \hat{t}} v)(t) = \frac{E_{\underline{p}}(t)}{\lamp(t)} \int \limits_0^\infty f_{\underline{p}}(s) \cdot v(s) \cdot \left(\hat{t} \wedge t \wedge s \right) \, \mathrm{d}s.
\end{equation}
Comparing \cite[Definition~4.7, Lemma~4.8]{RSW} with Notation~\ref{not:Hp} and~\eqref{eq:formulasss} we note that $\Hp(s)= \frac{E_{\underline{p}}(t) f_{\underline{p}}(t)}{\lamp(t)}$.

If we define the multiplication operator $\mathcal{D}_{\underline{p}}$ by $(\mathcal{D}_{\underline{p}}\varphi)(t):= \frac{E_{\underline{p}}(t)}{\lamp(t)} \cdot \varphi(t)$, then it is easy to check that we have $\Toperator{\hat{t}}=\mathcal{D}_{\underline{p}}^{-1} B_{{\underline{p}}, \hat{t}} \mathcal{D}_{\underline{p}}$ (cf.~equations~\eqref{eq:T_operator_def} and \eqref{eq:branching_op}). Thus, if we define $w:=\mathcal{D}_{\underline{p}}^{-1} v$ for a function $v\colon \mathbb{R}_+ \to \mathbb{R}$, then ${v \in L^2(\mathbb{R}_+,\, \rho_{\underline{p}})}$ if and only if ${w \in L^2(\mathbb{R}_+,\, \Hp)}$, and $B_{{\underline{p}}, \hat{t}}v = v$ is equivalent to ${\Toperator{\hat{t}}w = w}$. Therefore, the operators $B_{\underline{p},\hat{t}}$ and $\Toperator{\hat{t}}$ share the same eigenvalues.
The positivity of the principal eigenfunction is also preserved.
The operator $\Toperator{\hat{t}}$ is self-adjoint w.r.t.\ $L^2(\mathbb{R}_+,\, \Hp)$ (the proof is analogous to that of the self-adjointness of $B_{\underline{p},\hat{t}}$ w.r.t.~$L^2(\mathbb{R}_+,\, \rho_{\underline{p}})$, cf.\ \cite[Lemma~4.14~(i)]{RSW}).
The Hilbert--Schmidt norm of $\Toperator{\hat{t}}$ with respect to $L^2(\mathbb{R}_+,\, \Hp)$ is equal to the Hilbert--Schmidt norm of $B_{\underline{p},\hat{t}}$ with respect to $L^2(\mathbb{R}_+,\, \rho_{\underline{p}})$, cf.\ \cite[(A.1)]{RSW}.
Thus,~$\Toperator{\hat{t}}$ has the Hilbert--Schmidt property (even in the $\hat{t}=\infty$ case) by \cite[Lemma~4.14~(ii)]{RSW}. 

The $L^2(\mathbb{R}_+,\, \Hp)$ operator norm of $\Toperator{\hat{t}}$ is equal to the $L^2(\mathbb{R}_+,\, \rho_{\underline{p}})$ operator norm of $B_{\underline{p},\hat{t}}$ for every~$\hat{t}$, thus the remaining properties of the operator~$\Toperator{\hat{t}}$ listed in Theorem~\ref{thm:T_operator_prop} follow directly from the analogous properties established for~$B_{\underline{p},\hat{t}}$ in \cite[Theorem~1.45, Lemma~4.15]{RSW}.
\end{proof}

\begin{notation}[Almost $2$-regular degree constraint distributions]
Recall the set $\nonreg$ from~\eqref{eq:non_reg}. For any $\varepsilon_0 \in (0,1)$, we introduce
\begin{equation}\label{eq:near_reg_set}
\nearreg{\varepsilon_0} := \left\{ \, (p_2, \, p_3, \, \dots, \, p_\Delta ) \in \nonreg \; \Big| \; 1-p_2< \varepsilon_0 \, \right\}.
\end{equation}
\end{notation}

Recall the definition of $\tunder$ and $\tover$ from \eqref{eq:tc_approx}. The next lemma is the key to the proof of our main result.

\begin{lemma}[Lower and upper approximations of the principal eigenfunction]\label{lem:eigenfunction_bounds}
For any ${\delta \in (0, 1)}$, there exists ${\varepsilon_0>0}$ such that for any ${\underline{p} \in \nearreg{\varepsilon_0}}$, there exist non-negative functions~$\wunder(t)$ and~$\wover(t)$ such that for the operator $\Toperator{\hat{t}}$ defined in~\eqref{eq:T_operator_def}, we have
\begin{equation}\label{eq:goal_with_operator}
(\Toperator{\tunder}\wunder)(t) \le \wunder(t)
\quad \text{and} \quad (\Toperator{\tover}\wover)(t) \ge \wover(t) \quad \text{for any} \quad t \in \mathbb{R}_+.
\end{equation}
\end{lemma}

We postpone the proof of Lemma~\ref{lem:eigenfunction_bounds} to Section~\ref{subsec:eigenfunction_bounds}, and now we conclude Theorem~\ref{thm:tc_asymp_behavior_cont}.

\begin{proof}[Proof of Theorem~\ref{thm:tc_asymp_behavior_cont}]
First, we recall the following result of Perron--Frobenius theory for positive operators, see e.g.\ \cite[Section 6.5 Theorem 5.2]{M}. If~$T$ is a linear, Hilbert--Schmidt, positive operator and $w$ is a non-negative function such that
\begin{itemize}
\item $Tw \le \mu w$, then the principal eigenvalue of $T$ is at most $\mu$,
\item $Tw \ge \mu w$, then the principal eigenvalue of $T$ is at least $\mu$.
\end{itemize}
Therefore, we can deduce from statement~\ref{item:T_operator_sa_HS} of Theorem~\ref{thm:T_operator_prop} and Lemma~\ref{lem:eigenfunction_bounds} that for any $\delta \in (0, 1)$, there exists ${\varepsilon_0>0}$ such that for any $\underline{p} \in \nearreg{\varepsilon_0}$, the principal eigenvalue of~$\Toperator{\tunder}$ is at most~1 and the principal eigenvalue of~$\Toperator{\tover}$ is at least~1. 

Hence by statements~\ref{item:op_norm_1_at_tc} and~\ref{item:princip_eigenvalue} of Theorem~\ref{thm:T_operator_prop}, we obtain that
\begin{equation}\label{eq:op_norm_ineq}
\norm{\Toperator{\tunder}} \le \norm{\Toperator{\tc}} \le \norm{\Toperator{\tover}}.
\end{equation}

The inequalities~\eqref{eq:op_norm_ineq} together with statement \ref{item:op_norm_mon_in_t} of Theorem~\ref{thm:T_operator_prop} imply that
\begin{equation}\label{eq:goal_tc}
\tunder \le \tc \le \tover.
\end{equation}
Note that $\underline{p}^m \Rightarrow \chi_2$ is equivalent to $p_2^m \to 1$ as $m \to \infty$, thus for any $\delta\in (0,1)$ we have $\underline{p}^m \in \nearreg{\varepsilon_0}$ for large enough $m$. 
Therefore, by \eqref{eq:tc_approx}, the inequalities~\eqref{eq:goal_tc} imply Theorem~\ref{thm:tc_asymp_behavior_cont}.
\end{proof}

\subsection{Approximations of the principal eigenfunction}\label{subsec:eigenfunction_bounds}

In this section our goal is to prove Lemma~\ref{lem:eigenfunction_bounds}. We will explicitly construct the test functions~$\wunder$ and~$\wover$ that satisfy~\eqref{eq:goal_with_operator}.
For that purpose, we will use the first-order Taylor approximation of the principal eigenfunction of the operator~$\Toperator{\hat{t}}$ in terms of ${\varepsilon := 1-p_2}$ as $\varepsilon \to 0$. We will use directional derivatives, i.e., we let ${p_k = \indic{k=2} + \varepsilon \cdot r_k}$, cf.~Definition~\ref{def:deg_distr_representation}.
We will have to calculate the directional derivative of~$\Toperator{\hat{t}}$. Recalling~\eqref{eq:T_operator_def} and Notation~\ref{not:Hp}, we see that 
${(\Toperator{\hat{t}}\varphi)(t)=\ev [ \varphi(\tau_{\underline{p}}) \cdot (\hat{t} \wedge t \wedge \tau_{\underline{p}} ) ]}$, where $\tau_{\underline{p}}$ is a random variable with p.d.f.~$H_{\underline{p}}$.
Thus, we need to calculate the directional derivative of the function~$H_{\underline{p}}$ (cf.~Definition~\ref{def:H_approx} below). We use the first-order Taylor approximation of~$H_{\underline{p}}$ (as well as the survival function~$\Hint$ of $\tau_{\underline{p}}$) to explicitly construct auxiliary random variables $\tauunder{\delta}$ and~$\tauover{\delta}$ that serve as stochastic upper and lower bounds for~$\tau_{\underline{p}}$ (cf.~Lemma~\ref{lem:stoch_dom}).
Since the principal eigenfunction $w_{\underline{p}}$ of~$\Toperator{\tc}$ satisfies a second-order linear ODE on~${[0,\tc)}$ (cf.~Proposition~\ref{prop:principal_eigenfn_ODE}), we can calculate its directional derivative, too (cf.~Definition~\ref{def:w_approx}). In~Section~\ref{sss:asymp_prop} we will see that these derivatives can also be explicitly calculated. This allows us to explicitly construct the test functions~$\wunder$ and~$\wover$ (cf.~Definition~\ref{def:w_approx} below). 
Lemma~\ref{lem:sufficient_inequalities} states two inequalities that only involve the explicitly constructed random variables~$\tauunder{\delta}$ and~$\tauover{\delta}$ and the explicitly constructed test functions~$\wunder$ and~$\wover$ that imply the desired Lemma~\ref{lem:eigenfunction_bounds}.

\begin{definition}[Alternative representation of degree distributions]\label{def:deg_distr_representation}
For any $\underline{p} \in \nonreg$, we define
\begin{equation}\label{eq:eps_rk_def}
\varepsilon := 1-p_2,
\quad
r_2 := -1
\quad \text{and} \quad
r_k := \frac{p_k}{\varepsilon} \text{ for any } k = 3, \, \dots, \, \Delta.
\end{equation}

Note that $\sum_{d=3}^\Delta r_d = 1$, or equivalently, $\sum_{d=2}^\Delta r_d = 0$, and we have
\begin{equation}\label{eq:distribution_with_eps}
p_k = \indic{k=2} + \varepsilon \cdot r_k, \qquad k=2, \, 3, \, \dots, \, \Delta.
\end{equation}

We also introduce the set $\Rspace$ of vectors $\underline{r}=(r_k)_{k=2}^\Delta$ constructed via~\eqref{eq:eps_rk_def}:
\begin{equation}\label{eq:R_space}
\Rspace :=\left\{ \, (r_2,\, r_3, \, \dots, \, r_\Delta ) \; \Big| \; r_2 = -1,\; r_k \geq 0 \text{ for any } k=3, \, 4, \, \dots, \, \Delta\; \text{ and } \sum \limits_{d=3}^\Delta r_d = 1 \, \right\}.
\end{equation}

For any $\underline{r} \in \Rspace$, we define $r_k := 0$ for $k=1$ and for each $k > \Delta$.
\end{definition}

Definition~\ref{def:deg_distr_representation} immediately implies the following claim.

\begin{claim}[Bijection between representations]\label{cl:bijection}
Equations~\eqref{eq:eps_rk_def} and~\eqref{eq:distribution_with_eps} provide a bijection between~$\nonreg$ (cf.~\eqref{eq:non_reg}) and the set~$\left\{\,(\varepsilon, \underline{r}) \;|\; \varepsilon \in (0,1], \; \underline{r} \in \Rspace \, \right\}$. Moreover, for any $\varepsilon_0 \in (0,1)$, this gives rise to a bijection between~$\nearreg{\varepsilon_0}$ (cf.~\eqref{eq:near_reg_set}) and the set~$\left\{\,(\varepsilon, \underline{r}) \;|\; \varepsilon \in (0,\varepsilon_0), \; \underline{r} \in \Rspace \, \right\}$.
\end{claim}

Note that if $\varepsilon=0$, then for any $\underline{r} \in \Rspace$, equation~\eqref{eq:distribution_with_eps} gives the purely $2$-regular distribution $\chi_2$.

\begin{notation}[Distribution given by $\varepsilon$ and $\underline{r}$]\label{not:p(eps,r)}
For any~${\varepsilon \in (0,1]}$ and ${\underline{r} \in \Rspace}$, we denote by $\pfromr$ the element of~$\nonreg$ given by~\eqref{eq:distribution_with_eps}.
\end{notation}

Recall the notion of the functions~$\Hp$ and $H$ from Notation~\ref{not:Hp}. Now we study~$H_{\underline{p}(\varepsilon,\underline{r})}$ as $\varepsilon \to 0$.

\begin{definition}[First-order Taylor approximation of $\Hp$]\label{def:H_approx}
Let $\underline{r} \in \Rspace$. We define
\begin{equation}\label{eq:H_error_def}
\Herror(t) := \lim \limits_{\varepsilon \to 0} \frac{H_{\underline{p}(\varepsilon,\underline{r})}(t) - H(t)}{\varepsilon}
\quad \text{and} \quad
\Herrorint(t) := \int \limits_t^\infty \Herror(s) \, \mathrm{d}s,
\end{equation}
where $\underline{p}(\varepsilon, \underline{r}) \in \nonreg$ is defined in Notation~\ref{not:p(eps,r)}, i.e., $\Herror(t)$ is the directional derivative of $H_{\underline{p}(\varepsilon, \underline{r})}(t)$ in the direction $\underline{r}$ w.r.t.~$\varepsilon$ at~${\varepsilon=0}$.
We also introduce the notation 
\begin{equation*}
\Hint(t) := \int \limits_t^\infty \Hp(s) \, \mathrm{d}s\end{equation*} for any $\underline{p} \in \Pspace$, i.e., $\Hint$ denotes the survival function of the random variable~$\tau_{\underline{p}}$ (see~\eqref{eq:tau_def}). We again simplify our notation in the $2$-regular case and use $I(t):= I_{\chi_2}(t)$. Recall the notion of the counting process~$X_{\underline{p}}(t)$ from Notation~\ref{not:Hp}. Note that we have
\begin{equation}\label{eq:I(t)_explicit}
I(t) = \pr ( \tau_{\chi_2} \ge t) \stackrel{\eqref{eq:tau_def}}{=} \pr (X_{\chi_2}(t)=0) = \mathrm{e}^{-\lambda(t)},
\quad \text{therefore,}\quad
H(t) = -I'(t)= \lambda'(t) \cdot \mathrm{e}^{-\lambda(t)}.
\end{equation}

Furthermore, let $\varepsilon \in (0,1]$ and consider the distribution $\underline{p} = \pfromr \in \nonreg$ defined in Notation~\ref{not:p(eps,r)}. Then, for any $\delta \in (0,1)$, we define
\begin{equation}\label{eq:H_int_under_over_def}
\Hintunder{\delta}(t) := I(t) + (1-\delta)\varepsilon \cdot \Herrorint(t),
\qquad
\Hintover{\delta}(t) := I(t) + (1+\delta)\varepsilon \cdot \Herrorint(t).
\end{equation}
\end{definition}

Note that by the definition of~$\Herror$ (see~\eqref{eq:H_error_def}), the first-order Taylor approximation (in terms of $\varepsilon$) of~$\Hp(t)$ is ${H(t) + \varepsilon \cdot \Herror(t)}$. However, by the fact that $\Hp$ and $H$ are both probability density functions and by~\eqref{eq:H_error_def}, we see that
\begin{equation}\label{eq:H_error_int_at_0}
\Herrorint(0) = \int \limits_0^\infty \Herror(s) \, \mathrm{d}s=0,
\end{equation}
which also implies that the sign of~$\Herror(t)$ changes. On the other hand, we will show that $\Herrorint(t)$ is non-negative for any ${t \in \mathbb{R}_+}$ (see Claim~\ref{cl:H_error_explicit}, where we give an explicit formula for $\Herrorint(t)$). This property will be beneficial for us, thus, instead of approximating $\Hp$, we will use the approximation ${\Hint(t) \approx I(t) + \varepsilon \cdot \Herrorint(t)}$, i.e., the approximation of the survival function of~$\tau_{\underline{p}}$ (see~\eqref{eq:tau_def}). Namely, we will show in Lemma~\ref{lem:H_int_squeeze} that~$\Hintunder{\delta}(t)$ and~$\Hintover{\delta}(t)$ squeeze~$\Hint(t)$ for any~${t \in [0,\toverone]}$ (where~$\toverone$ was defined in~\eqref{eq:tc_approx}). This result will imply Lemma~\ref{lem:stoch_dom} that will in turn be an ingredient in the proof of Lemma~\ref{lem:eigenfunction_bounds}.

Recall that we say that the random variable $Y$ stochastically dominates the random variable $X$ (or briefly denote $X \preccurlyeq Y$) if $\mathbb{P}(X \geq t) \leq \mathbb{P}(Y \geq t)$ holds for all $t \in \mathbb{R}$. In this case $\mathbb{E}(f(X)) \leq \mathbb{E}(f(Y))$ also holds for any monotone increasing function $f: \mathbb{R} \to \mathbb{R}$.

We now stochastically dominate $\tau_{\underline{p}}$ from above and from below with a pair of auxiliary random variables that are easier to work with than $\tau_{\underline{p}}$. 

\begin{lemma}[Stochastic dominance for $\tau_{\underline{p}}$]\label{lem:stoch_dom}
For any $\delta \in (0,1)$, there exists $\varepsilon_0>0$ such that for any ${\underline{p} \in \nearreg{\varepsilon_0}}$, we can define random variables $\tauunder{\delta}$ and $\tauover{\delta}$ such that for any $t \in [0,\toverone]$, we have
\begin{equation}\label{eq:tau_pm_def}
\pr \left( \tauunder{\delta} \ge t \right) =\Hintunder{\delta}(t),
\qquad
\pr \left( \tauover{\delta} \ge t \right) =\Hintover{\delta}(t)
\end{equation}
(cf.\ \eqref{eq:H_int_under_over_def}),
and they satisfy the stochastic dominance relations
\begin{equation}\label{eq:stoch_dom}
\tauunder{\delta} \preccurlyeq \tau_{\underline{p}} \preccurlyeq \tauover{\delta}.
\end{equation}
\end{lemma}

We prove Lemma~\ref{lem:stoch_dom} in Section~\ref{sss:stoch_dom}.

Now, we can turn to the explicit definitions of the functions $\wunder(t)$ and $\wover(t)$ that appeared in Lemma~\ref{lem:eigenfunction_bounds}.
Recall that the principal eigenvalue of $\Toperator{\tc}$ is $\norm{\Toperator{\tc}}=1$ (see statements \ref{item:op_norm_1_at_tc} and \ref{item:princip_eigenvalue} of Theorem~\ref{thm:T_operator_prop}). Therefore, by \cite[Lemma~4.17]{RSW} we have the following result.

\begin{proposition}[ODE of the principal eigenfunction]\label{prop:principal_eigenfn_ODE}
Let $\underline{p} \in \Pspace$. Let us define the function~$w_{\underline{p}}(t)$ as follows. On the interval $[0,\tc)$ it is the solution of the second-order linear initial value problem
\begin{equation}\label{eq:w_p_ivp}
w''_{\underline{p}}(t) = -\Hp(t)w_{\underline{p}}(t), \quad w_{\underline{p}}(0) =0, \quad w'_{\underline{p}}(0) =1,
\end{equation}
and if $t \ge \tc$ then we define $w_{\underline{p}}(t):=w_{\underline{p}}(\tc)$. Then $w_{\underline{p}}$ is the unique principal eigenfunction of~$\Toperator{\tc}$ (corresponding to the principal eigenvalue $1$), normalized in a way that $w'_{\underline{p}}(0) = 1$. In particular, $w_{\underline{p}}(t) \geq 0$ for all $t \in \mathbb{R}_+$.
\end{proposition}

\begin{claim}[Principal eigenfunction in the $2$-regular case]\label{cl:2reg_eigenfunction}
In the $2$-regular case, we have $w_{\chi_2}(t) = \lambda(t)$ for all $t \geq 0$, i.e., the principal eigenfunction of $T_{\chi_2, \hat{t}_c(\chi_2)}$ is the solution of the initial value problem \eqref{eq:lambda_2reg_ode}.
\end{claim}

\begin{proof}
We know that $\hat{t}_c(\chi_2)=+\infty$ (see Remark~\ref{rem:2reg_hattc_infty}). Differentiating the ODE \eqref{eq:lambda_2reg_ode} of $\lambda(t)$ and using the formula \eqref{eq:I(t)_explicit} for $H(t)$ we obtain $\lambda''(t)=-H(t)\lambda(t)$. The identities $\lambda(0)=0$ and $\lambda'(0)=1$ also follow from~\eqref{eq:lambda_2reg_ode}. Thus $w_{\chi_2}(t) = \lambda(t)$ by \eqref{eq:w_p_ivp}.
\end{proof}

We now introduce the directional derivative $\werror(t):=\lim \limits_{\varepsilon \to 0} \frac{w_{\underline{p}(\varepsilon,\underline{r})}(t) - \lambda(t)}{\varepsilon}$ and construct the test functions~$\wunder$ and~$\wover$ that we promised in the statement of Lemma~\ref{lem:eigenfunction_bounds}.
Recall $\tunder$ and $\tover$ from~\eqref{eq:tc_approx}.

\begin{definition}[Approximations of the principal eigenfunction]\label{def:w_approx}
Given some $\underline{r} \in \Rspace$, we define the function~$\werror(t)$ to be the solution of the following initial value problem:
\begin{equation}\label{eq:w_error_ivp}
\werror''(t) = -H(t)\cdot \werror(t)-\Herror(t) \cdot \lambda(t), \quad\werror(0) = 0, \quad \werror'(0) = 0.
\end{equation}

Let $\varepsilon \in (0,1]$ and consider the probability distribution $\underline{p}=\pfromr \in \nonreg$ defined in Notation~\ref{not:p(eps,r)}. For any $\delta \in (0, 1)$, we define the functions $\wunder(t)$ and $\wover(t)$ by
\begin{equation}\label{eq:w_approx}
\wunder(t):= \left\{
\begin{array}{@{}ll@{}}
\lambda(t) +\left(1+\frac{\delta}{2}\right)\varepsilon\werror(t) \quad &\text{if } t \le \tunder,\\
\wunder(\tunder) \quad &\text{if } t > \tunder,
\end{array}\right.
\quad \text{and} \quad
\wover(t):= \left\{
\begin{array}{@{}ll} \lambda(t) + \left(1-\frac{\delta}{2}\right)\varepsilon\werror(t) \quad &\text{if } t \le \tover,\\
\wover(\tover) \quad &\text{if } t > \tover.
\end{array}\right.
\end{equation}
\end{definition}

Note that bijection of Claim~\ref{cl:bijection} implies that~$\wunder(t)$ and~$\wover(t)$ are defined for any $\underline{p} \in \nonreg$.

The second-order linear ODE~\eqref{eq:w_error_ivp} can be solved explicitly, cf.\ Claim~\ref{cl:w_error_explicit}.

We will show that the test functions~$\wunder$, $\wover$ defined in~\eqref{eq:w_approx} satisfy the requirements of Lemma~\ref{lem:eigenfunction_bounds}. Recall Claim~\ref{cl:2reg_eigenfunction} and note that the initial value problem~\eqref{eq:w_error_ivp} is obtained taking the directional derivative of~\eqref{eq:w_p_ivp} w.r.t.~$\varepsilon$ in the direction $\underline{r}$ and substituting ${\varepsilon=0}$. Therefore, the first-order Taylor approximation (in terms of~$\varepsilon$) of $w_{\underline{p}}(t)$ on the interval~${(0,\tc)}$ is ${\lambda(t) +\varepsilon \cdot \werror(t)}$. We use this approximation (with an extra~${1\pm \frac{\delta}{2}}$) to define~$\wunder(t)$ for~${t \le \tunder}$ and~$\wover(t)$ for~${t \le \tover}$. For larger values of~$t$, the functions~$\wunder(t)$ and~$\wover(t)$ are defined to be constants, since $(\Toperator{\hat{t}}\varphi)(t)$ is also constant for~${t \ge \hat{t}}$ for any~${\varphi:\mathbb{R}_+ \to \mathbb{R}}$ (cf.~\eqref{eq:T_operator_def}).

\begin{proposition}[Properties of the approximations of the eigenfunction]\label{prop:w_approx_prop}
Let $\delta \in (0, 1)$ be arbitrary. Then
\begin{enumerate}[label={\upshape(\alph*)}]
\item for any $\underline{p} \in \nonreg$, we have $\wunder(0)=0$, $\wover(0)=0$, and
\item there exists $\varepsilon_0>0$ such that for any $\underline{p} \in \nearreg{\varepsilon_0}$, the functions $\wunder(t)$ and $\wover(t)$ are monotone increasing on~$\mathbb{R}_+$.
\end{enumerate}
\end{proposition}

We prove Proposition~\ref{prop:w_approx_prop} in Section~\ref{sss:asymp_prop}.
Finally, to prove Lemma~\ref{lem:eigenfunction_bounds}, we state two inequalities.

\begin{lemma}[Inequalities for the approximations of the eigenfunction]\label{lem:sufficient_inequalities}
For any $\delta \in (0, 1)$, there exists~${\varepsilon_0>0}$ such that for any $\underline{p} \in \nearreg{\varepsilon_0}$, the functions $\wunder(t)$, $\wover(t)$ (cf.~\eqref{eq:w_approx}) and the random variables~$\tauunder{\delta/3}$, $\tauover{\delta/3}$ (cf.~\eqref{eq:tau_pm_def}) satisfy the inequalities
\begin{align}
\ev \left( \wunder(\tauover{\delta/3}) \cdot \indic{t < \tauover{\delta/3}} \right) \le \wunder'(t) \quad \text{ for any } \quad t < \tunder, \label{eq:w_under_derivative_goal}\\
\ev \left( \wover(\tauunder{\delta/3}) \cdot \indic{t < \tauunder{\delta/3}} \right) \ge \wover'(t) \quad \text{ for any } \quad t < \tover\label{eq:w_over_derivative_goal}.
\end{align}
\end{lemma}

We prove Lemma~\ref{lem:sufficient_inequalities} in Section~\ref{sss:inequalities}. Now we can conclude the proof of Lemma~\ref{lem:eigenfunction_bounds}.

\begin{proof}[Proof of Lemma~\ref{lem:eigenfunction_bounds}]
First, recall that $\wunder(t)$ and $\wover$ are non-negative by Proposition~\ref{prop:w_approx_prop}.

Since $\wunder(t)$ and $\Toperator{\tunder}\wunder(t)$ are constant for $t \ge \tunder$, it is enough to prove the first inequality of~\eqref{eq:goal_with_operator} for $t \le \tunder$. Similarly, in order to show the second inequality of~\eqref{eq:goal_with_operator}, it is enough to consider the case~${t \le \tover}$. By~\eqref{eq:T_operator_def} and the fact that $\Hp$ is the p.d.f.\ of $\tau_{\underline{p}}$ (see Notation~\ref{not:Hp}), we need to show that
\begin{equation}\label{eq:goal_with_expectation}
\ev \left( \wunder(\tau_{\underline{p}}) \cdot (\tau_{\underline{p}} \wedge t) \right) \le \wunder(t) \text{ for any } t \le \tunder
\quad \text{and} \quad
\ev \left( \wover(\tau_{\underline{p}}) \cdot (\tau_{\underline{p}} \wedge t) \right) \ge \wover(t) \text{ for any } t \le \tover.
\end{equation}

Since \eqref{eq:goal_with_expectation} holds at $t=0$ (see Proposition~\ref{prop:w_approx_prop}), it is enough to show that the inequalities hold for the derivatives w.r.t.~$t$, i.e.,
\begin{align}
\begin{split}\label{eq:goal_with_derivative}
&\ev \left( \wunder(\tau_{\underline{p}}) \cdot \indic{t < \tau_{\underline{p}}} \right) \le \wunder'(t) \; \text{ for any } \; t < \tunder,\\
&\ev \left( \wover(\tau_{\underline{p}}) \cdot \indic{t < \tau_{\underline{p}}} \right) \ge \wover'(t) \; \text{ for any } \; t < \tover.
\end{split}
\end{align}

Observe that the functions in the expectations in \eqref{eq:goal_with_derivative} are monotone increasing in $\tau_{\underline{p}}$ (see Proposition~\ref{prop:w_approx_prop}). Therefore, by the stochastic domination relations \eqref{eq:stoch_dom}, Lemma~\ref{lem:sufficient_inequalities} implies~\eqref{eq:goal_with_derivative} if $1-p_2 < \varepsilon_0$ (where $\varepsilon_0$ is the minimum of the $\varepsilon_0$ values that appear in Lemma~\ref{lem:stoch_dom}, Proposition~\ref{prop:w_approx_prop} and Lemma~\ref{lem:sufficient_inequalities}). The proof of Lemma~\ref{lem:eigenfunction_bounds} is complete.
\end{proof}

The rest of the paper is dedicated to the deferred proofs of Lemma~\ref{lem:stoch_dom}, Proposition~\ref{prop:w_approx_prop} and Lemma~\ref{lem:sufficient_inequalities}. In Section~\ref{sss:study_lambda}, we study the function~$\lamp(t)$ (introduced in Notation~\ref{not:lambda}), and establish bounds that squeeze~$\lamp(t)$ using first-order Taylor approximations. In Section~\ref{sss:stoch_dom}, we study the functions introduced in Definition~\ref{def:H_approx} to construct the random variables $\tauunder{\delta}$, $\tauover{\delta}$ required for the stochastic dominance result in Lemma~\ref{lem:stoch_dom}. In Section~\ref{sss:asymp_prop}, we study the asymptotic properties of the functions used in the construction of $\wunder$ and $\wover$
and prove Proposition~\ref{prop:w_approx_prop}.
Finally, in Section~\ref{sss:inequalities}, we combine the established explicit formulas and the asymptotic relations to prove Lemma~\ref{lem:sufficient_inequalities}.

\subsubsection{Approximation of the function \texorpdfstring{$\lamp(t)$}{lambda}}\label{sss:study_lambda}

The goal of Section~\ref{sss:study_lambda} is to prove Lemma~\ref{lem:lambda_p_squeeze} which states that the function $\lamp$ (that was introduced in Notation~\ref{not:lambda}) can be squeezed between the functions $\lamunder{\delta}$ and $\lamover{\delta}$ on the interval $[0,\toverone]$, where $\lamunder{\delta}$ and~$\lamover{\delta}$ are defined in~\eqref{eq:lambda_under_over_def} below.
Loosely speaking, the functions $\lamunder{\delta}$ and~$\lamover{\delta}$ arise as the first-order Taylor approximation of $\lamp$ in terms of the variable $\varepsilon$ (with the extra $(1 \pm \delta)$ terms making sure that the squeeze actually happens). In Claim~\ref{cl:properties_lamerror} we provide explicit formulas for 
$\lamunder{\delta}$ and $\lamover{\delta}$.

\begin{definition}[First-order Taylor approximation of $\lamp$]
Let $\underline{r} \in \Rspace$ (cf.~\eqref{eq:R_space}). We define
\begin{equation}\label{eq:lam_error_def}
\lamerror(t) := \lim \limits_{\varepsilon \to 0} \frac{\lambda_{\pfromr}(t) - \lambda(t)}{\varepsilon},
\end{equation}
where $\pfromr$ is defined in Notation~\ref{not:p(eps,r)}.

Let $\varepsilon \in (0,1]$ and consider the distribution~$\underline{p} = \pfromr \in \nonreg$. For any $\delta \in (0,1)$, we define
\begin{equation}\label{eq:lambda_under_over_def}
\lamunder{\delta}(t) := \lambda(t) + \left( 1-\delta \right) \varepsilon \cdot \lamerror(t), \qquad \lamover{\delta}(t) := \lambda(t) + \left( 1+\delta \right) \varepsilon \cdot \lamerror(t).
\end{equation}
\end{definition}

By the definition of~$\lamerror$ (see~\eqref{eq:lam_error_def}), the first-order Taylor approximation (in terms of~$\varepsilon$) of the function~$\lamp(t)$ is~${\lambda(t) + \varepsilon \cdot \lamerror(t)}$. We will show that for any $\delta \in (0,1)$, the functions $\lamunder{\delta}(t)$ and $\lamover{\delta}(t)$ squeeze~$\lamp(t)$ on the interval $[0,\, \toverone]$ if~$\varepsilon$ is small enough (see Lemma~\ref{lem:lambda_p_squeeze}). Before that, we introduce some notation.

\begin{definition}[Tail of $\underline{r}$]
Let $\underline{r} \in \Rspace$ (cf.\ \eqref{eq:R_space}). We define $s_k$, $k \in \mathbb{N}_+$ as
\begin{equation}\label{eq:sk_def}
s_k := \sum \limits_{d=k}^\Delta r_d
\quad \text{for} \quad
k=2,\, 3,\, \dots, \, \Delta
\qquad \text{and} \qquad
s_k :=0
\quad \text{if} \quad
k=1 \; \text{ or } \; k > \Delta.
\end{equation}
Observe that for any $\underline{r} \in \Rspace$, we have $s_2 = 0$ and $s_3=1$.
\end{definition}

We note that for any~${\underline{p} \in \nonreg}$, we can define $(q_k)_{k=2}^\Delta$ as in~\eqref{eq:qk_def}, $\varepsilon$ and $\underline{r} =(r_k)_{k=2}^\Delta$ as in~\eqref{eq:eps_rk_def} and $(s_k)_{k=2}^\Delta$ as in~\eqref{eq:sk_def} for this $\underline{r}=(r_k)_{k=2}^\Delta$. Observe that we have $q_k=\varepsilon \cdot s_k$ for $k = 3,\, \dots,\, \Delta$. In particular, the ODE~\eqref{eq:lambda_p_ode_with_q} for the function~$\lamp(t)$ can be equivalently rewritten as
\begin{equation}\label{eq:lambda_p_ode_with_s}
\lambda'_{\underline{p}}(t) = \mathrm{e}^{-\lamp(t)} \cdot \left( 1+\lamp(t)+ \varepsilon \cdot \sum \limits_{k=2}^{\Delta-1} \frac{\lambda^k_{\underline{p}}(t)}{k!} \cdot s_{k+1} \right), \qquad \lamp(0) = 0.
\end{equation}

Now we study the function~$\lamerror(t)$.

\begin{claim}[Properties of $\lamerror$]\label{cl:properties_lamerror}
Let $\underline{r} \in \Rspace$, and consider the function $\lamerror(t)$ defined in~\eqref{eq:lam_error_def}. It solves the initial value problem
\begin{equation}\label{eq:lam_error_ode}
\lamerror'(t) = \mathrm{e}^{-\lambda(t)} \cdot \left( -\lambda(t) \cdot \lamerror(t)+ \sum \limits_{k=2}^{\Delta-1} \frac{\lambda^k(t)}{k!} \cdot s_{k+1} \right), \qquad \lamerror(0) = 0,
\end{equation}
where $(s_k)_{k=3}^\Delta$ is defined in~\eqref{eq:sk_def}. Furthermore, we have
\begin{equation}\label{eq:lam_error_solution}
\lamerror(t) = \lambda'(t) \cdot \int \limits_0^{t} \frac{1}{1+\lambda(s)} \cdot \left( \sum \limits_{k=2}^{\Delta-1} \frac{\lambda^k(s)}{k!} \cdot s_{k+1}\right) \, \mathrm{d}s = \lambda'(t) \cdot \int \limits_0^{\lambda(t)} \frac{\mathrm{e}^x}{(1+x)^2} \cdot \left( \sum \limits_{k=2}^{\Delta-1} \frac{x^k}{k!} \cdot s_{k+1}\right) \, \mathrm{d}x.
\end{equation}

Moreover, for any $t \in \mathbb{R}_+$, we have
\begin{equation}\label{eq:lam_error_bounds}
0 \le \lamerror(t) \le \frac 12 \sum \limits_{k=1}^{\Delta-2} \frac{\lambda^k(t)}{k!}s_{k+2}.
\end{equation}
\end{claim}

\begin{proof}
The initial value problem~\eqref{eq:lam_error_ode} is obtained by differentiating \eqref{eq:lambda_p_ode_with_s} w.r.t.~$\varepsilon$ and substituting $\varepsilon=0$. The formula in~\eqref{eq:lam_error_solution} is the unique solution of the first-order linear ODE~\eqref{eq:lam_error_ode}.

The non-negativity of $\lamerror(t)$ follows from \eqref{eq:lam_error_solution} and the fact that $\lambda'(t) \ge 0$ (cf.\ \eqref{eq:lambda_2reg_ode}).

Finally, we prove the upper bound in \eqref{eq:lam_error_bounds}. First, observe that $\lamerror(0) = \frac 12 \sum_{k=1}^{\Delta-2} \frac{\lambda^k(0)}{k!}s_{k+2}=0$ (cf.~\eqref{eq:lam_error_ode} and~\eqref{eq:lambda_2reg_ode}). Therefore, it is enough to prove that if there is equality at the right-hand side of~\eqref{eq:lam_error_bounds} at some~${t \in \mathbb{R}_+}$, then the inequality holds for the derivatives at~$t$.
Observe that the function on the right-hand side of~\eqref{eq:lam_error_bounds} is monotone increasing, since it is the sum of monotone increasing functions (cf.~\eqref{eq:lambda_2reg_ode}). Therefore, it is enough to prove that
\begin{equation}\label{eq:forbidden_region}
\text{if } {\lamerror(t) = \frac 12 \sum_{k=1}^{\Delta-2} \frac{\lambda^k(t)}{k!}s_{k+2}} \text{ at some } t \in \mathbb{R}_+, \text{ then we have } \lamerror'(t) \le 0.
\end{equation}
So let us assume that we have ${\lamerror(t) = \frac 12 \sum_{k=1}^{\Delta-2} \frac{\lambda^k(t)}{k!}s_{k+2}}$ at some $t \in \mathbb{R}_+$. Then we obtain~\eqref{eq:forbidden_region}:
\begin{equation*}
\lamerror'(t) \stackrel{\eqref{eq:lam_error_ode}}{=}\mathrm{e}^{-\lambda(t)} \cdot \sum \limits_{k=2}^{\Delta-1} \frac{\lambda^k(t)}{k!} \cdot s_{k+1} \cdot \left(1-\frac k2 \right) \le 0.
\end{equation*}
This completes the proof of~\eqref{eq:lam_error_bounds} and, consequently, the proof of Claim~\ref{cl:properties_lamerror}.
\end{proof}

The following bounds will turn out to be useful later.

\begin{claim}[Bounds on $\lambda(t)$ and $\lamover{\delta}(t)$]\label{cl:lambda_bounds}
For any $t \in \mathbb{R}_+$, we have
\begin{equation}\label{eq:lam_grows_logarithmic}
\ln (t+1) \le \lambda(t) \le 2\ln(2t+1)
\quad \text{and} \quad
\lambda'(t) \ge \frac{1}{t+1}.
\end{equation}

Moreover, for any $\delta \in (0,1)$ and $k \in \mathbb{N}_+$, there exists $\varepsilon_0>0$ such that for any $\underline{p} \in \nearreg{\varepsilon_0}$ and for any $t \in [0,\toverone]$, we have
\begin{equation}\label{eq:lambda_universal_bounds}
\mathrm{(a)} \quad \lamover{\delta}(t) \le 3\ln \left( \frac{2}{\varepsilon}\right),
\qquad \quad
\mathrm{(b)}\quad \varepsilon \cdot \lambda^k(t)\le \delta,
\qquad \quad
\mathrm{(c)}\quad (\lamover{\delta}(t))^k \le \lambda^k(t) \cdot \left(1 + 2^k \cdot \frac{2\varepsilon\lamerror(t)}{\lambda(t)} \right),
\end{equation}
where $\varepsilon$ and $\underline{r}=(r_k)_{k=2}^\Delta$ are defined in~\eqref{eq:eps_rk_def}.
\end{claim}

\begin{proof}
The solution of the initial value problem \eqref{eq:lambda_2reg_ode}
is $\lambda(t) = \psi^{-1}(t)$, where $\psi(\lambda) := \int_0^\lambda \frac{\mathrm{e}^s}{1+s}\, \mathrm{d}s$. Observe that for any $\lambda \in \mathbb{R}_+$, we have
\begin{equation}\label{eq:psi_bounds}
\psi(\lambda) \le \int \limits_0^\lambda \mathrm{e}^s \, \mathrm{d}s = \mathrm{e}^{\lambda}-1
\qquad \text{and} \qquad
\psi(\lambda) \ge \int \limits_0^\lambda \frac{\mathrm{e}^s}{1+\lambda}\, \mathrm{d}s = \frac{\mathrm{e}^{\lambda}-1}{1+\lambda} \ge \frac 12 \left( \mathrm{e}^{\lambda/2}-1\right),
\end{equation}
where in the last inequality we use that
\begin{equation*}
\mathrm{e}^{\lambda}-1 = \left( \mathrm{e}^{\lambda/2}-1\right) \cdot \left( \mathrm{e}^{\lambda/2}+1\right) \ge
\left( \mathrm{e}^{\lambda/2}-1\right) \cdot \left( 2 + \frac{\lambda}{2}\right). 
\end{equation*}

Therefore, by inverting the bounds in \eqref{eq:psi_bounds}, we obtain that $\ln(t+1) \le \psi^{-1}(t) \le 2\ln(2t+1)$, which is equivalent to the bounds on $\lambda(t)$ in \eqref{eq:lam_grows_logarithmic}.

By \eqref{eq:lambda_2reg_ode}, $\lambda'(t) \ge 1/(t+1)$ is equivalent to 
\begin{equation}\label{eq:lam_deri_lower_bound_equivalent}
(1+t)(1+\lambda(t)) - \mathrm{e}^{\lambda(t)} \ge 0.
\end{equation}

Observe that \eqref{eq:lam_deri_lower_bound_equivalent} holds at $t=0$ and the derivative of the left-hand side is $(1+t)\lambda'(t) \ge 0$. Hence, we obtain that \eqref{eq:lam_deri_lower_bound_equivalent} is satisfied for any $t \in\mathbb{R}_+$, i.e., the proof of \eqref{eq:lam_grows_logarithmic} is complete.

In order to prove the inequalities in \eqref{eq:lambda_universal_bounds}, observe that
\begin{equation*}
\toverone \stackrel{\eqref{eq:tc_approx}}{\le} \frac{2}{3(1-p_2)} = \frac{2}{3\varepsilon}.
\end{equation*}
Thus, the upper bound in~\eqref{eq:lam_grows_logarithmic} gives that for any $\underline{p}\in \nearreg{2/3}$, we have
\begin{equation}\label{eq:lam_tover_bound}
\lambda(t) \le 2\ln\left( \frac{2}{\varepsilon}\right) \quad \text{for any} \quad t \in [0,\toverone].
\end{equation}

Together with~\eqref{eq:lambda_under_over_def} and~\eqref{eq:lam_error_bounds}, it implies~\eqref{eq:lambda_universal_bounds}(a).
The inequality~\eqref{eq:lambda_universal_bounds}(b) follows from \eqref{eq:lam_tover_bound} and the fact that $\varepsilon \cdot (\ln (1/\varepsilon))^k \to 0$ as $\varepsilon \to 0$ for any $k \in \mathbb{N}_+$. In order to prove \eqref{eq:lambda_universal_bounds}(c), observe that by \eqref{eq:lam_error_bounds} and \eqref{eq:lambda_universal_bounds}(b), $\varepsilon \cdot \lamerror(t)/\lambda(t) <1$ on the interval $[0,\toverone]$ if $\varepsilon$ is small enough. Therefore, \eqref{eq:lambda_universal_bounds}(c) follows from the fact that for any $x \in [0,1]$ and $k \in \mathbb{N}_+$, we have $(1+x)^k \le 1+2^k \cdot x$.
\end{proof}

Besides the bounds stated in Claim~\ref{cl:lambda_bounds}, it is also useful to note that for any $x \ge 0$, we have
\begin{equation}\label{eq:exp_bound}
1-x \le \mathrm{e}^{-x} \le 1-x+\frac{x^2}{2}.
\end{equation}

Now we are ready to prove that $\lamunder{\delta}(t)$ and $\lamover{\delta}(t)$ (both defined in \eqref{eq:lambda_under_over_def}) squeeze $\lamp(t)$ on $[0,\toverone]$.

\begin{lemma}[Bounds on $\lamp$]\label{lem:lambda_p_squeeze}
For any $\delta \in (0,1)$, there exists $\varepsilon_0>0$ such that for any $\underline{p} \in \nearreg{\varepsilon_0}$, we have
\begin{equation}\label{eq:lambda_p_bounds}
\lamunder{\delta}(t) \le \lamp(t) \le \lamover{\delta}(t) \qquad \text{for any } \qquad t \le \toverone.
\end{equation}
\end{lemma}

\begin{proof}
First, we prove the lower bound stated in \eqref{eq:lambda_p_bounds}. Since $\lamp(0) = \lamunder{\delta}(0) = 0$ (see \eqref{eq:lambda_p_ode_with_s}, \eqref{eq:lambda_under_over_def} and~\eqref{eq:lam_error_ode}), it is enough to prove that
\begin{equation}\label{eq:assumption_lam_under}
\text{if } \lamunder{\delta}(t)=\lamp(t) \text{ at some } t \in [0,\toverone], \text{ then we have } \lamunder{\delta}'(t) \le \lambda'_{\underline{p}}(t).
\end{equation}

For that purpose, observe that
\begin{align}
\begin{split}\label{eq:lam_under_ode}
\lamunder{\delta}'(t) \stackrel{\eqref{eq:lambda_under_over_def}}{=}&
\lambda'(t) + \left( 1-\delta \right) \varepsilon \cdot \lamerror'(t)
\stackrel{\eqref{eq:lambda_2reg_ode},\, \eqref{eq:lam_error_ode}}{=}\\
&\mathrm{e}^{-\lambda(t)} \cdot \left[1+\lambda(t) + (1-\delta)\varepsilon \cdot \left( -\lambda(t) \cdot \lamerror(t)+ \sum \limits_{k=2}^{\Delta-1} \frac{\lambda^k(t)}{k!} \cdot s_{k+1} \right)\right].
\end{split}
\end{align}

Therefore, if $\lamunder{\delta}(t)=\lamp(t)$ at some $t \in [0,\toverone]$, then we have
\begin{align}
\begin{split}\label{eq:lam_under_derivative_difference}
\lamp'(t)
\stackrel{\eqref{eq:lambda_p_ode_with_s}}{=}&
\mathrm{e}^{-\lamunder{\delta}(t)} \cdot \left( 1+\lamunder{\delta}(t)+ \varepsilon \cdot \sum \limits_{k=2}^{\Delta-1} \frac{\lamunder{\delta}^k(t)}{k!} \cdot s_{k+1} \right) 
\stackrel{\eqref{eq:lambda_under_over_def}, \, \eqref{eq:exp_bound}}{\ge}\\
&\mathrm{e}^{-\lambda(t)} \cdot \left(1-(1-\delta)\varepsilon \cdot \lamerror(t)\right) \left( 1+\lambda(t) + \left( 1-\delta \right) \varepsilon \cdot \lamerror(t)+ \varepsilon \cdot \sum \limits_{k=2}^{\Delta-1} \frac{\lambda^k(t)}{k!} \cdot s_{k+1} \right)
\stackrel{\eqref{eq:lam_under_ode}}{\ge}\\
&\lamunder{\delta}'(t)-\varepsilon^2 \cdot \mathrm{e}^{-\lambda(t)} \left( \lamerror^2(t) + \lamerror(t) \sum \limits_{k=2}^{\Delta-1} \frac{\lambda^k(t)}{k!} \cdot s_{k+1} \right)+
\delta \cdot \varepsilon \cdot\mathrm{e}^{-\lambda(t)} \cdot \frac{\lambda^2(t)}{2}.
\end{split}
\end{align}
Using~\eqref{eq:lam_error_bounds} and~\eqref{eq:lambda_universal_bounds}(b), inequality~\eqref{eq:lam_under_derivative_difference} implies~\eqref{eq:assumption_lam_under}, and hence we obtain the lower bound in~\eqref{eq:lambda_p_bounds}.

We prove the upper bound in~\eqref{eq:lambda_p_bounds} using a similar argument as in the case of the lower bound, i.e., we show that
\begin{equation}\label{eq:assumption_lam_over}
\text{if } \lamover{\delta}(t)=\lamp(t) \text{ at some } t \in [0,\toverone], \text{ then we have } \lamover{\delta}'(t) \ge \lambda'_{\underline{p}}(t).
\end{equation}

Analogously to~\eqref{eq:lam_under_ode} and~\eqref{eq:lam_under_derivative_difference}, using also~\eqref{eq:lambda_universal_bounds}(a) and~\eqref{eq:lambda_universal_bounds}(c), we obtain that if~${\lamover{\delta}(t) = \lamp(t)}$ at some ${t \in [0,\toverone]}$, then we have
\begin{equation}\label{eq:lam_over_derivative_difference}
\lambda'_{\underline{p}}(t) \le
\lamover{\delta}'(t)
+ \varepsilon^2 \cdot \mathrm{e}^{-\lambda(t)} \cdot \left[ \frac{2\lamerror(t)}{\lambda(t)} \cdot \sum \limits_{k=2}^{\Delta-1} \lambda^k(t)+8\lamerror^2(t)\cdot \ln \left( \frac{2}{\varepsilon}\right) \right]
- \delta \cdot \varepsilon \cdot \mathrm{e}^{-\lambda(t)} \cdot \frac{\lambda^2(t)}{2}.
\end{equation}

Similarly to the lower bound, \eqref{eq:lam_over_derivative_difference} implies \eqref{eq:assumption_lam_over}. This completes the proof of~\eqref{eq:lambda_p_bounds}.
\end{proof}

\subsubsection{Stochastic dominance}\label{sss:stoch_dom}

In Section~\ref{sss:stoch_dom} we prove Lemma~\ref{lem:stoch_dom}, thus we will need to study the functions~$\Herror(t)$, $\Hint(t)$, $\Herrorint(t)$, $\Hintunder{\delta}(t)$ and $\Hintover{\delta}(t)$ that were introduced in Definition~\ref{def:H_approx}.
In order to prove that we can define~$\tauunder{\delta}$ and~$\tauover{\delta}$ as in~\eqref{eq:tau_pm_def}, we need to show that their survival functions~$\Hintunder{\delta}(t)$ and~$\Hintover{\delta}(t)$ are monotone decreasing on the interval~$[0,\toverone]$. 
For the purpose of proving $\tauunder{\delta} \preccurlyeq \tau_{\underline{p}} \preccurlyeq \tauover{\delta}$, we will show in Lemma~\ref{lem:H_int_squeeze} that the corresponding survival functions~$\Hintunder{\delta}$ and $\Hintover{\delta}$ squeeze~$\Hint(t)$ on $[0, \toverone]$.
In order to achieve these goals, we need to study the functions~$\Herror(t)$ and~$\Herrorint(t)$. First, we give explicit formulas for them.

Recall the definition of the function~$\lamerror(t)$ from~\eqref{eq:lam_error_def}, the ODE~\eqref{eq:lam_error_ode} for~$\lamerror(t)$ and the explicit integral formula of~$\lamerror(t)$ from \eqref{eq:lam_error_solution}. Also recall the notation~$s_k$ from~\eqref{eq:sk_def}.

\begin{claim}[Explicit formulas for~$\Herror(t)$ and~$\Herrorint(t)$]\label{cl:H_error_explicit}
For any $\underline{r} \in \Rspace$, we have
\begin{align}
\Herror(t) &= \mathrm{e}^{-\lambda(t)} \cdot \left( \lamerror'(t) - \lambda'(t) \cdot \lamerror(t) + \lambda'(t) \cdot \sum \limits_{k=0}^{\Delta-2} \frac{\lambda^k(t)}{k!}r_{k+2}\right),\label{eq:H_error_formula}\\
\Herrorint(t) &= \mathrm{e}^{-\lambda(t)} \cdot \left( -\lamerror(t) + \sum \limits_{k=1}^{\Delta-2} \frac{\lambda^k(t)}{k!}s_{k+2}\right)\label{eq:H_error_int}.
\end{align}
\end{claim}

Before we prove Claim~\ref{cl:H_error_explicit}, let us note that by~\eqref{eq:lam_error_bounds}, equation~\eqref{eq:H_error_int} shows that $\Herrorint(t)$ is a strictly positive function on~$\mathbb{R}_+$, and that this property will be crucial in the proof of the stochastic dominance.

\begin{remark}[Probabilistic interpretation of~\eqref{eq:H_error_int}]
Before we rigorously prove~\eqref{eq:H_error_int}, we briefly sketch its probabilistic meaning. Let~$\tau_{\underline{p}}$, $X_{\underline{p}}(t)$ and $D_{\underline{p}}$ be as in Notation~\ref{not:Hp} for the distribution ${\underline{p}=\pfromr}$ defined in Notation~\ref{not:p(eps,r)} with some small~$\varepsilon>0$. Then, by Notation~\ref{not:Hp} and Definition~\ref{def:H_approx}, we have ${\Herrorint(t) = \frac{\mathrm{d}}{\mathrm{d}\varepsilon} \pr (\tau_{\underline{p}} \ge t) \big|_{\varepsilon=0}}$. Observe that if $\varepsilon$ is small enough, then the event $\{\tau_{\underline{p}} \ge t\}$ is approximately the union of the events $\{X_{\underline{p}}(t) < 1 \}$ and $\{ X_{\chi_2}(t) < D_{\underline{p}}-1 \}$. Now calculating the probabilities of these two events, differentiating them w.r.t.~$\varepsilon$ and substituting ${\varepsilon=0}$, we obtain \eqref{eq:H_error_int}.
\end{remark}

\begin{proof}[Proof of Claim~\ref{cl:H_error_explicit}]
It was shown in \cite[Proof of Lemma~4.8]{RSW} that for any $\underline{p} \in \nonreg$, we have
\begin{equation}\label{eq:H_eps}
\Hp(t) = \lamp'(t)\mathrm{e}^{-\lamp(t)} \cdot \sum \limits_{k=0}^{\Delta-2} \frac{\lamp^k(t)}{k!} \cdot p_{k+2} \stackrel{\eqref{eq:distribution_with_eps}}{=} \lamp'(t)\mathrm{e}^{-\lamp(t)} \cdot \left( 1 + \varepsilon \cdot \sum \limits_{k=0}^{\Delta-2} \frac{\lamp^k(t)}{k!} \cdot r_{k+2}\right).
\end{equation}

By differentiating \eqref{eq:H_eps} w.r.t.~$\varepsilon$ and substituting $\varepsilon=0$, we obtain \eqref{eq:H_error_formula}.

Now we prove \eqref{eq:H_error_int}. Since equation \eqref{eq:H_error_int} is satisfied at $t=0$ (cf.~\eqref{eq:H_error_int_at_0}, \eqref{eq:lambda_2reg_ode} and \eqref{eq:lam_error_ode}), it is enough to prove that the derivatives of the two sides of \eqref{eq:H_error_int} are equal. We know that $\Herrorint'(t) = -\Herror(t)$ 
by the definition of~$\Herrorint$. On the other hand, the derivative of the right-hand side of \eqref{eq:H_error_int} is also $-\Herror(t)$ by \eqref{eq:H_error_formula}. These considerations imply \eqref{eq:H_error_int}.
\end{proof}

In order to prove the stochastic dominance stated in Lemma~\ref{lem:stoch_dom}, we use that the functions~$\Hintunder{\delta}(t)$ and~$\Hintover{\delta}(t)$ squeeze~$\Hint(t)$. More precisely, we show the following lemma.

\begin{lemma}[Bounds on $\Hint$]\label{lem:H_int_squeeze}
For any $\delta \in (0,1)$, there exists $\varepsilon_0>0$ such that for any $\underline{p} \in \nearreg{\varepsilon_0}$, we have
\begin{equation}\label{eq:H_int_bounds}
\Hintunder{\delta}(t) \le \Hint(t) \le \Hintover{\delta}(t) \qquad \text{for any} \qquad t \le \toverone.
\end{equation}
\end{lemma}

Before proving Lemma~\ref{lem:H_int_squeeze}, we derive Lemma~\ref{lem:stoch_dom} from it.

\begin{proof}[Proof of Lemma~\ref{lem:stoch_dom}]
First, we show that we can define $\tauunder{\delta}$ and $\tauover{\delta}$ as in \eqref{eq:tau_pm_def}. Observe that both $\Hintunder{\delta}(t)$ and $\Hintover{\delta}(t)$ are equal to $1$ at $t=0$ (cf.\ \eqref{eq:H_int_under_over_def}, \eqref{eq:I(t)_explicit}, \eqref{eq:H_error_int_at_0}). Therefore, in order to show that we can define~$\tauunder{\delta}$ and $\tauover{\delta}$ as in \eqref{eq:tau_pm_def}, it is enough to prove that if $\varepsilon>0$ is small enough, then $\Hintunder{\delta}(t)$ and $\Hintover{\delta}(t)$ are monotone decreasing in $t$ on the interval $[0,\toverone]$. By Definition~\ref{def:H_approx}, we obtain that
\begin{equation}\label{eq:H_int_bounds_derivatives}
\Hintunder{\delta}'(t) = -\mathrm{e}^{-\lambda(t)} \cdot \left( \lambda'(t) + (1-\delta)\varepsilon \cdot \Herror(t) \right)
\quad \text{and} \quad
\Hintover{\delta}'(t) = -\mathrm{e}^{-\lambda(t)} \cdot \left( \lambda'(t) + (1+\delta)\varepsilon \cdot \Herror(t) \right).
\end{equation}

By \eqref{eq:H_error_formula}, \eqref{eq:lam_error_ode}, \eqref{eq:lam_error_bounds} and \eqref{eq:lambda_universal_bounds}(b), we obtain that the functions in \eqref{eq:H_int_bounds_derivatives} are negative on the interval~${[0,\toverone]}$ if $\varepsilon$ is small enough, i.e., the functions~$\Hintunder{\delta}(t)$ and~$\Hintover{\delta}(t)$ are monotone decreasing.

Finally, we prove that even~\eqref{eq:stoch_dom} can be satisfied. Note that~$\tauunder{\delta}$ can be defined such that
\begin{equation}\label{eq:tau_minus_def_at_toverone}
\pr \left( \tauunder{\delta} = \toverone \right) =\Hintunder{\delta}(\toverone),
\quad\text{i.e.,}\quad
\pr \left( \tauunder{\delta} > \toverone \right) =0.
\end{equation}

Moreover, by~\eqref{eq:H_int_bounds}, $\tauover{\delta}$ can be defined such that
\begin{equation}\label{eq:tau_plus_def_at_toverone}
\pr \left( \tauover{\delta} = \toverone \right) =\Hintover{\delta}(\toverone)-\Hint(\toverone) 
\qquad \text{and} \qquad
\pr \left( \tauover{\delta} \ge t \right) =\Hint(t) \quad \text{for any} \quad t > \toverone.
\end{equation}

Combining~\eqref{eq:H_int_bounds} with \eqref{eq:tau_pm_def}, \eqref{eq:tau_minus_def_at_toverone} and \eqref{eq:tau_plus_def_at_toverone} yields $\mathbb{P}(\tauunder{\delta}\geq t) \leq \mathbb{P}(\tau_{\underline{p}}\geq t) \leq \mathbb{P}(\tauover{\delta}\geq t)$ for all $t \geq 0$, i.e., the stochastic dominance relations~\eqref{eq:stoch_dom} hold.
\end{proof}

It remains to prove Lemma~\ref{lem:H_int_squeeze}. For that purpose, we study the function $\Hint(t)$, i.e., the survival function of $\tau_{\underline{p}}$. Recall~$q_k$ from~\eqref{eq:qk_def}. Note that it was shown in \cite[Lemma~4.8~(b)]{RSW} that for any $\underline{p} \in \Pspace$, we have
\begin{equation}\label{eq:H_p_int_with_lambda}
\Hint(t) = 1- \int \limits_0^t \Hp(s)\, \mathrm{d}s = \Hintlam(\lamp(t)),
\text{ where }
\Hintlam(\lambda) := \sum \limits_{k=0}^{\Delta-2} \mathrm{e}^{-\lambda} \frac{\lambda^k}{k!} \cdot q_{k+2} = \mathrm{e}^{-\lambda} \left( 1+ \varepsilon \sum \limits_{k=1}^{\Delta-2} \frac{\lambda^k}{k!} \cdot s_{k+2} \right).
\end{equation}

Observe that $\Hintlam(\lambda) = \ev (q_{X_\lambda+2})$, where $X_\lambda \sim \text{POI}(\lambda)$.

\begin{claim}[Properties of the function $\Hintlam$]\label{cl:H_int_lam_prop}
There exists $\varepsilon_0>0$ such that for any $\underline{p}\in \nearreg{\varepsilon_0}$ the function~$\Hintlam(\lambda)$ is monotone decreasing on $\mathbb{R}_+$ and convex on the interval $\left[0, \sqrt[\Delta]{1/(2\varepsilon \Delta)} \right]$, where $\varepsilon:=1-p_2$. Moreover,
\begin{equation}\label{eq:Hintlam_diff_bound}
\Hintlam'(\lambda) \le -\frac12 \mathrm{e}^{-\lambda} \qquad \text{for any} \qquad \lambda \in \mathbb{R}_+.
\end{equation}
\end{claim}

\begin{proof}
By differentiating \eqref{eq:H_p_int_with_lambda} w.r.t.~$\lambda$, and using that $s_2=s_{\Delta+1}=r_{\Delta+1}=0$ (see Definition~\ref{def:deg_distr_representation} and~\eqref{eq:sk_def}),
we obtain
\begin{equation}\label{eq:H_int_lam_derivatives}
\Hintlam'(\lambda) = -\mathrm{e}^{-\lambda} \left( 1+ \varepsilon \sum \limits_{k=0}^{\Delta-2} \frac{\lambda^k}{k!} \cdot r_{k+2} \right)
\quad \text{and} \quad
\Hintlam''(\lambda) = \mathrm{e}^{-\lambda} \left( 1+ \varepsilon \sum \limits_{k=0}^{\Delta-2} \frac{\lambda^k}{k!} \cdot (r_{k+2}-r_{k+3}) \right).
\end{equation}

In the formula for $\Hintlam'(\lambda)$ only the term corresponding to $k=0$ is positive (since $r_2=-1$ but $r_k \geq 0$ if $k \geq 3$, cf.~\eqref{eq:R_space}), hence if $\varepsilon <1$, then $\Hintlam'(\lambda) <0$ for any $\lambda \in \mathbb{R}_+$, i.e., $\Hintlam(\lambda)$ is monotone decreasing. Moreover, if $\varepsilon \le 1/2$, then \eqref{eq:Hintlam_diff_bound} also holds.

In order to show the convexity of $\Hintlam(\lambda)$, observe that for any ${\varepsilon \le 1/(2\Delta)}$ and ${\lambda \in [0, \sqrt[\Delta]{1/(2\varepsilon \Delta)}]}$, we have ${2\varepsilon\sum_{k=0}^{\Delta-2}\lambda^k \le 1}$. Hence, by~\eqref{eq:H_int_lam_derivatives}, $\Hintlam(\lambda)$ is convex on the interval~${[0, \sqrt[\Delta]{1/(2\varepsilon \Delta)}]}$ if~$\varepsilon$ is small enough.
\end{proof}

Now we can conclude the proof of Lemma~\ref{lem:H_int_squeeze}, i.e., that $\Hintunder{\delta}$ and $\Hintover{\delta}$ squeeze $\Hint(t)$ on $[0, \toverone]$.
In the proof we will make use of Lemma~\ref{lem:lambda_p_squeeze}, i.e., that $\lamunder{\delta}$ and $\lamover{\delta}$ squeeze $\lamp$ on $[0,\toverone]$.

\begin{proof}[Proof of Lemma~\ref{lem:H_int_squeeze}]
Let us fix $\delta \in (0,1)$. First, we show the lower bound for $\Hint(t)$. We have
\begin{align}
\begin{split}\label{eq:H_int_lower_close}
\Hintlam(\lamover{\delta}(t)) 
\stackrel{\substack{\eqref{eq:lambda_under_over_def}, \, \eqref{eq:H_p_int_with_lambda},\\ \eqref{eq:exp_bound}}}{\ge}
&\mathrm{e}^{-\lambda(t)} \cdot \left( 1 -(1+\delta)\varepsilon\lamerror(t)\right) \left(1+ \varepsilon \sum \limits_{k=1}^{\Delta-2} \frac{\lambda^k(t)}{k!}s_{k+2}\right)
\stackrel{\substack{\eqref{eq:H_int_under_over_def}, \, \eqref{eq:I(t)_explicit},\\ \eqref{eq:H_error_int}}}{=}\\
&\Hintunder{\delta}(t)+\mathrm{e}^{-\lambda(t)} \cdot \left[\delta \cdot \varepsilon \left( \sum \limits_{k=1}^{\Delta-2} \frac{\lambda^k(t)}{k!}s_{k+2} - 2\lamerror(t) \right)
- 2 \varepsilon^2 \lamerror(t) \cdot \sum \limits_{k=1}^{\Delta-2} \frac{\lambda^k(t)}{k!}s_{k+2} \right]
\stackrel{\eqref{eq:lam_error_bounds}}{\ge}\\
&\Hintunder{\delta}(t) - \mathrm{e}^{-\lambda(t)} \cdot 2 \varepsilon^2 \cdot \lamerror(t) \cdot \sum \limits_{k=1}^{\Delta-2} \frac{\lambda^k(t)}{k!}.
\end{split}
\end{align}

Note that by \eqref{eq:lambda_universal_bounds}(a), if $\varepsilon$ is small enough, then $\lamover{\delta}(t)\le\sqrt[\Delta]{1/(2\varepsilon\Delta)}$ for any $t \in [0,\toverone]$. Therefore, the function~$\Hintlam(\lamover{\delta}(t))$ is convex on the interval $t \in [0,\toverone]$ (see Claim~\ref{cl:H_int_lam_prop}). Hence, using the first-order Taylor expansion of the function $\Hintlam(\lambda)$, we have
\begin{align}
\begin{split}\label{eq:H_int_lower_Taylor}
\Hintlam(\lamover{\delta/2}(t)) - \Hintlam(\lamover{\delta}(t)) \ge
&-\Hintlam'(\lamover{\delta}(t)) \cdot (\lamover{\delta}(t) - \lamover{\delta/2}(t) ) \stackrel{\eqref{eq:Hintlam_diff_bound}, \, \eqref{eq:lambda_under_over_def}}{\ge} \frac12 \cdot\mathrm{e}^{-\lamover{\delta}(t)} \cdot \frac{\delta}{2}\cdot \varepsilon \cdot \lamerror(t)\ge\\ 
&\frac 18\cdot\mathrm{e}^{-\lambda(t)} \cdot \delta \cdot \varepsilon \cdot \lamerror(t),
\end{split}
\end{align}
where in the last inequality we used that if $\varepsilon$ is small enough, then we have
\begin{equation*}
\mathrm{e}^{-\lamover{\delta}(t)}
\stackrel{\eqref{eq:lambda_under_over_def}, \, \eqref{eq:exp_bound}}{\ge}
\mathrm{e}^{-\lambda(t)} \cdot (1-(1+\delta)\varepsilon \lamerror(t))
\stackrel{\eqref{eq:lam_error_bounds}, \, \eqref{eq:lambda_universal_bounds}\mathrm{(b)}}{\ge}
\frac 12 \mathrm{e}^{-\lambda(t)}.
\end{equation*}

Hence, using the monotonicity of $\Hintlam(\lambda)$, we obtain the lower bound for $\Hint(t)$:
\begin{equation*}
\Hintunder{\delta}(t) \stackrel{\eqref{eq:H_int_lower_close}, \, \eqref{eq:H_int_lower_Taylor},\, \eqref{eq:lambda_universal_bounds}\mathrm{(b)}}{\le}
\Hintlam(\lamover{\delta/2}(t)) \stackrel{\eqref{eq:lambda_p_bounds}}{\le}
\Hintlam(\lamp(t)) \stackrel{\eqref{eq:H_p_int_with_lambda}}{=}
\Hint(t).
\end{equation*}

We can derive the upper bound of \eqref{eq:H_int_bounds} similarly. Analogously to \eqref{eq:H_int_lower_close} and \eqref{eq:H_int_lower_Taylor}, using also the inequalities~\eqref{eq:lambda_universal_bounds}(a) and \eqref{eq:lambda_universal_bounds}(c), we obtain that if $\varepsilon$ is small enough, then on the interval $[0,\toverone]$, we have
\begin{equation}\label{eq:H_int_upper_close} 
\Hintlam(\lamunder{\delta}(t)) 
\le
\Hintover{\delta}(t) + 2\mathrm{e}^{-\lambda(t)} \cdot \varepsilon^2 \cdot \lamerror(t) \cdot \left(2\lamerror(t) + \sum \limits_{k=0}^{\Delta-3}\lambda^k(t) \right)
\end{equation}
and
\begin{equation}\label{eq:H_int_upper_Taylor}
\Hintlam(\lamover{\delta}(t)) - \Hintlam(\lamover{\delta/2}(t)) \ge
\frac 18\cdot\mathrm{e}^{-\lambda(t)} \cdot \delta \cdot \varepsilon \cdot \lamerror(t).
\end{equation}

Therefore, we obtain the upper bound of \eqref{eq:H_int_bounds} if $\varepsilon$ is small enough:
\begin{equation*}
\Hintover{\delta}(t) \stackrel{\eqref{eq:H_int_upper_close}, \, \eqref{eq:H_int_upper_Taylor}, \, \eqref{eq:lambda_universal_bounds}\mathrm{(b)}}{\ge}
\Hintlam(\lamunder{\delta/2}(t)) \stackrel{\eqref{eq:lambda_p_bounds}}{\ge} 
\Hintlam(\lamp(t)) \stackrel{\eqref{eq:H_p_int_with_lambda}}{=}
\Hint(t).\qedhere
\end{equation*}
\end{proof}

\subsubsection{Asymptotic properties of the approximations of the eigenfunction}\label{sss:asymp_prop}

In Section~\ref{sss:asymp_prop} our goal is to study the asymptotic properties of the test functions $\wunder$ and $\wover$ that appear in the statement of Lemma~\ref{lem:sufficient_inequalities} that we will conclude in Section~\ref{sss:inequalities}. E.g.\ in Claim~\ref{cl:w_error_explicit} we will give an explicit formula for $\werror$ (introduced in Definition~\ref{def:w_approx}), and in Claims~\ref{cl:w_error_bounds} and~\ref{cl:asymp} we determine the growth rate of all of the terms that appear the explicit formula for $\werror$ and~$\werror'$.
As a side-product of these analytic calculations, we also deduce Proposition~\ref{prop:w_approx_prop}, the main ingredient of the proof of which is the following lemma. The statement of Lemma~\ref{lem:monotonicity} contains a constant~$\overline{K}$: in the proof of Proposition~\ref{prop:w_approx_prop} we will choose~${\overline{K}=2}$, but in Section~\ref{sss:inequalities} we will need to choose a bigger~${\overline{K}=\overline{K}(\delta)}$. 

\begin{lemma}[Bound on the slope of the approximations of the eigenfunction]\label{lem:monotonicity}
For any $\overline{K} \in \mathbb{R}_+$, there exists $\varepsilon_0>0$ such that if $\underline{p} \in \nearreg{\varepsilon_0}$, then on the interval $t \in [0, \toverone]$, we have
\begin{equation}\label{eq:w_error_derivative_goal}
\varepsilon \cdot \overline{K} \cdot |\werror'(t)| \le \lambda'(t),
\end{equation}
where the $\varepsilon$ and $\underline{r}=(r_k)_{k=2}^\Delta$ corresponding to $\underline{p}$ are defined in~\eqref{eq:eps_rk_def}.
\end{lemma}

Before proving Lemma~\ref{lem:monotonicity}, we show that it implies Proposition~\ref{prop:w_approx_prop}.

\begin{proof}[Proof of Proposition~\ref{prop:w_approx_prop}]
Let $\delta \in (0,1)$ and $\underline{p} \in \nonreg$.

\begin{enumerate}[label=(\alph*)]
\item Observe that the equations $\wunder(0)=\wover(0)=0$ follow from \eqref{eq:lambda_2reg_ode} and \eqref{eq:w_error_ivp}.
\item Inequality \eqref{eq:w_error_derivative_goal} (with $\overline{K}=2$) implies that there exists $\varepsilon_0>0$ such that if $\underline{p} \in \nearreg{\varepsilon_0}$, then
\begin{itemize}
\item $\wunder'(t) = \lambda'(t) + \left(1+\frac{\delta}{2}\right)\varepsilon \werror'(t) \ge 0$ on the interval $[0,\tunder)$,
\item $\wover'(t) = \lambda'(t) + \left(1-\frac{\delta}{2}\right)\varepsilon \werror'(t) \ge 0$ on the interval $[0,\tover)$.
\end{itemize}

Recall that the functions~$\wunder(t)$ and~$\wover(t)$ are constant for~${t \ge \tunder}$ and for~${t \ge \tover}$, respectively (cf.~\eqref{eq:w_approx}). Therefore, $\wunder(\cdot)$ and $\wover(\cdot)$ are monotone increasing on $\mathbb{R}_+$ if $\varepsilon$ is small enough.\qedhere
\end{enumerate}
\end{proof}

In order to prove Lemma~\ref{lem:monotonicity}, we need to study the function $\werror(t)$. We begin by giving an explicit formula for it.
Recall the function $\Herror(t)$ from \eqref{eq:H_error_def}.

\begin{claim}[Explicit formula for $\werror$]\label{cl:w_error_explicit}
For any $\underline{r}\in \Rspace$, the function $\werror(t)$ can be expressed as
\begin{equation}\label{eq:w_error_ivp_solution}
\werror(t) = \werrorfirst(t) \cdot \lambda(t) + \werrorsecond(t) \cdot u(t),
\end{equation}
where
\begin{equation}\label{eq:w_error_ivp_solution_aux_functions}
u(t) := 1+\lambda(t) \cdot \int \limits_0^t \frac{\lambda'(s) - 1}{\lambda^2(s)} \, \mathrm{d}s,
\quad
\werrorfirst(t) := -\int \limits_0^t u(s) \cdot \Herror(s) \cdot \lambda(s)\, \mathrm{d}s,
\quad
\werrorsecond(t) := \int \limits_0^t \Herror(s) \cdot \lambda^2(s)\, \mathrm{d}s.
\end{equation}

Furthermore, we have
\begin{equation}\label{eq:werror_u_derivative}
\werror'(t) = 
\werrorfirst(t)\lambda'(t) + \werrorsecond(t)u'(t),
\quad \text{where} \quad
u'(t) = 
\lambda'(t) \cdot \int \limits_0^t \frac{\lambda'(s) - 1}{\lambda^2(s)} \, \mathrm{d}s + \frac{\lambda'(t) - 1}{\lambda(t)}.
\end{equation}
\end{claim}

\begin{proof}
Note that by l'H\^{o}pital's rule and \eqref{eq:lambda_2reg_ode}, we obtain that $(\lambda'(t)-1)/\lambda^2(t) \to -1/2$ as $t \searrow 0$. Since~$\lambda(t)$ is monotone increasing and $\lambda'(t) \in [0,1]$ for any $t \in \mathbb{R}_+$, it follows that the function $(\lambda'(t)-1)/\lambda^2(t)$ is bounded. Thus the function $u(t)$ in \eqref{eq:w_error_ivp_solution_aux_functions} is well-defined. Therefore, the function on the right-hand side of~\eqref{eq:w_error_ivp_solution} is also well-defined, and it is straightforward to check that it solves the initial value problem~\eqref{eq:w_error_ivp}.

The formulas in \eqref{eq:werror_u_derivative} follow from \eqref{eq:w_error_ivp_solution} and \eqref{eq:w_error_ivp_solution_aux_functions} by differentiation.
\end{proof}

\begin{remark}[Derivation of~\eqref{eq:w_error_ivp_solution}]
Using that $\lambda(t)$ is a solution of the ODE ${y''(t) = -H(t) \cdot y(t)}$ (cf.\ the proof of Claim~\ref{cl:2reg_eigenfunction}), we obtain $u(t)$ as a linearly independent solution of this homogeneous second-order differential equation using Abel's identity. Then, using the method of variation of parameters, we can find the general solution of the inhomogeneous differential equation~\eqref{eq:w_error_ivp}, too. Finally, substituting the initial values, we obtain that the solution of the initial value problem~\eqref{eq:w_error_ivp} can be written as in~\eqref{eq:w_error_ivp_solution}.
\end{remark}

\begin{notation}[Molloy--Reed constant]
Let $\underline{r} \in \Rspace$. We introduce the notation
\begin{equation}\label{eq:Upsilon_r_def}
\Upsilon_{\underline{r}} := \sum \limits_{k=3}^\Delta k(k-2)r_k. 
\end{equation}

Observe that with this notation, for any $\underline{p} \in \nonreg$, we have $\tunder = (1-\delta)/(\varepsilon \cdot \Upsilon_{\underline{r}})$ and $\tover = (1+\delta)/(\varepsilon \cdot \Upsilon_{\underline{r}})$, where the $\varepsilon$ and $\underline{r}=(r_k)_{k=2}^\Delta$ corresponding to $\underline{p}$ are defined in~\eqref{eq:eps_rk_def}.
\end{notation}

Recall from Remark~\ref{rem:molloy_reed} that the RDCP and the configuration model asymptotically have the same critical edge density in the almost $2$-regular limit. 
This is why we call $\Upsilon_{\underline{r}}$ the Molloy--Reed constant, but let us stress again that these two random graph models are genuinely different (cf.~Remark~\ref{rem:flaws}), thus our main result (Theorem~\ref{thm:tc_asymp_behavior_cont}) does not follow from the classical Molloy--Reed criterion.

Before we prove Lemma~\ref{lem:monotonicity}, we derive some useful bounds and asymptotic relations on the various terms that appear in the explicit formulas for $\werror$ and $\werror'$ in Claims~\ref{cl:w_error_bounds} and~\ref{cl:asymp}.

\begin{claim}[Bounds on $\Herror$ and $\werrorsecond$]\label{cl:w_error_bounds}
For any $\underline{r} \in \Rspace$, we have
\begin{enumerate}[label={\upshape(\alph*)}]
\item $|\Herror(t)| \le 2\lambda'(t) \mathrm{e}^{-\lambda(t)} \cdot \sum_{k=0}^{\Delta-1} \frac{\lambda^k(t)}{k!}$ for any $t \in \mathbb{R}_+$, \puteqnum\label{eq:H_error_abs_bound}
\item $| \werrorsecond(t) | \le 2\Delta^2 \cdot (\Delta+1)$ for any $t \in \mathbb{R}_+$.\puteqnum\label{eq:w_error_2nd_bound}
\end{enumerate}

Furthermore,
\begin{enumerate}
\item[$\mathrm{(c)}$] $\lim \limits_{t \to \infty} \werrorsecond(t)= \Upsilon_{\underline{r}}$, uniformly in $\underline{r} \in \Rspace$. \puteqnum\label{eq:w_error_2nd_limit}
\end{enumerate}
\end{claim}

\needspace{3\baselineskip}
\begin{proof}
\mbox{}
\begin{enumerate}[label=(\alph*)]
\item First, we prove \eqref{eq:H_error_abs_bound}.
If we rewrite $\Herror(t)$ with the help of \eqref{eq:H_error_formula} and \eqref{eq:lam_error_ode}, then we can obtain an upper and a lower bound for $\Herror(t)$ by dropping the negative and the positive terms, respectively. Using also the bounds in~\eqref{eq:lam_error_bounds} in the lower bound, we can derive that
\begin{equation}\label{eq:H_error_bounds}
\Herror(t) \le 2\lambda'(t) \mathrm{e}^{-\lambda(t)} \cdot \sum \limits_{k=1}^{\Delta-1} \frac{\lambda^k(t)}{k!},
\qquad
\Herror(t) \ge -\lambda'(t) \mathrm{e}^{-\lambda(t)} \cdot \left( \sum \limits_{k=1}^{\Delta-2} \frac{\lambda^k(t)} {k!}+1 \right).
\end{equation}
The inequalities in \eqref{eq:H_error_bounds} imply \eqref{eq:H_error_abs_bound}.

\item In order to prove the statements related to the function $\werrorsecond(t)$, we introduce the notation $F_k(t)$ for the cumulative distribution function of the gamma distribution with shape $k+1$ and scale $1$, i.e.,
\begin{equation}\label{eq:gamma_cdf}
F_k(t) := \int \limits_0^t \mathrm{e}^{-x} \cdot \frac{x^k}{k!} \, \mathrm{d}x.
\end{equation}

Observe that by substituting $\lambda = \lambda(s)$, we obtain that for any $t \in \mathbb{R}_+$, we have
\begin{equation}\label{eq:gamma_integral}
\int \limits_0^t \lambda'(s) \cdot \mathrm{e}^{-\lambda(s)} \cdot \frac{\lambda^k(s)}{k!} \, \mathrm{d}s \stackrel{\eqref{eq:gamma_cdf}}{=} F_k(\lambda(t)).
\end{equation}

Hence, we obtain \eqref{eq:w_error_2nd_bound}:
\begin{equation*}
|\werrorsecond(t)| \stackrel{\eqref{eq:w_error_ivp_solution_aux_functions}, \, \eqref{eq:H_error_abs_bound}, \, \eqref{eq:gamma_integral}}{\le} \sum \limits_{k=0}^{\Delta-1} 2(k+2)(k+1)F_{k+2}(\lambda(t)) \le 2\Delta^2 \cdot (\Delta+1).
\end{equation*}

\item Finally, we prove \eqref{eq:w_error_2nd_limit}. By substituting $\lambda=\lambda(s)$, changing the order of integrations and by the fact ${\lim_{t \to \infty} F_k(t) = 1}$ (cf.\ \eqref{eq:gamma_cdf}), we obtain
\begin{multline}\label{eq:change_integrals}
\int \limits_0^\infty \mathrm{e}^{-\lambda(s)} \cdot \lambda(s) \cdot \lambda'(s) \cdot \lamerror(s) \, \mathrm{d}s
\stackrel{\eqref{eq:lam_error_solution}, \, \eqref{eq:lambda_2reg_ode}}{=}
\int \limits_0^\infty \mathrm{e}^{-2\lambda} \cdot (1+\lambda) \cdot \lambda \cdot \int \limits_0^\lambda \frac{\mathrm{e}^x}{(1+x)^2} \cdot \sum \limits_{k=2}^{\Delta-1} \frac{x^k}{k!}s_{k+1} \, \mathrm{d}x \, \mathrm{d}\lambda=\\
 \sum \limits_{k=2}^{\Delta-1} \left[s_{k+1} \int \limits_0^\infty \frac{\mathrm{e}^x}{(1+x)^2} \cdot \frac{x^k}{k!} \int \limits_x^\infty \mathrm{e}^{-2\lambda} \cdot(1+\lambda) \cdot \lambda \, \mathrm{d}\lambda \, \mathrm{d}x\right] \stackrel{\eqref{eq:gamma_cdf}}{=}
\sum \limits_{k=2}^{\Delta-1} s_{k+1} \cdot \frac12 = \frac 12 \sum \limits_{k=3}^\Delta (k-2)r_k.
\end{multline}

On the other hand, the function~${\mathrm{e}^{-\lambda(t)} \cdot \lamerror(t) \cdot \lambda^2(t)}$ converges to~$0$ as ${t \to \infty}$ (see~\eqref{eq:lam_error_bounds} and~\eqref{eq:lam_grows_logarithmic}). Therefore, using integration by parts, we obtain
\begin{align}
\begin{split}\label{eq:H_error_lambda^2_1st_component}
\int \limits_0^\infty \mathrm{e}^{-\lambda(s)} \cdot \left( \lamerror'(s) - \lambda'(s) \cdot \lamerror(s) \right) \cdot \lambda^2(s) \, \mathrm{d}s
\stackrel{\eqref{eq:change_integrals}}{=}&
\left[ \mathrm{e}^{-\lambda(s)} \cdot \lamerror(s) \cdot \lambda^2(s) \right]_{s=0}^\infty - \sum \limits_{k=3}^\Delta (k-2)r_k =\\
&-\sum \limits_{k=3}^\Delta (k-2)r_k.
\end{split}
\end{align}

Noting that $r_2 = -\sum_{k=3}^\Delta r_k$ (see~\eqref{eq:R_space}), we can conclude \eqref{eq:w_error_2nd_limit}:
\begin{equation*}
\lim \limits_{t \to \infty} \werrorsecond(t)
\stackrel{\eqref{eq:w_error_ivp_solution_aux_functions},\, \eqref{eq:H_error_formula},\, \eqref{eq:gamma_integral},\, \eqref{eq:H_error_lambda^2_1st_component}}{=}
\sum \limits_{k=3}^\Delta \left[-(k-2)+(k^2-k-2)\right] \cdot r_k\stackrel{\eqref{eq:Upsilon_r_def}}{=} \Upsilon_{\underline{r}}.
\end{equation*}
\end{enumerate}

Since $\Rspace$ is compact, for each $t \in \mathbb{R}_+$ the function~$\werrorsecond(t)$ and the limit
$\lim \limits_{t \to \infty} \werrorsecond(t) = \int_0^\infty \Herror(s) \cdot \lambda^2(s)\, \mathrm{d}s$ depend continuously on~$\underline{r}$, moreover, $\werrorsecond(t)$ converges to $\Upsilon_{\underline{r}}$ monotonically, we can use Dini's theorem to conclude that this convergence is uniform in~${\underline{r} \in \Rspace}$.
\end{proof}

We also need to study the asymptotic behavior as $t \to \infty$ of some functions defined earlier.

Recall the notion of $h(t) \sim g(t)$ and $h(t) \ll g(t)$ from Notation~\ref{not:asymp_relations}.

\begin{notation}[Asymptotic uniform boundedness in $\underline{r}$]
Let $h_{\underline{r}}\colon \mathbb{R}_+ \to \mathbb{R}$ be a function that may depend on $\underline{r} \in \Rspace$ and $g \colon \mathbb{R}_+ \to \mathbb{R}_+$ be a non-negative function. We use the notation ${h_{\underline{r}}(t) = \mathcal{O}(g(t))}$ if there exist constants $C,\, T \in \mathbb{R}_+$ such that
\begin{equation*}
\sup \limits_{\underline{r} \in \Rspace} |h_{\underline{r}}(t)| \le C \cdot g(t)
\quad \text{for any} \quad t \ge T.
\end{equation*}
\end{notation}

\begin{claim}[Asymptotic behavior of frequently used functions]\label{cl:asymp}
Let $\underline{r} \in \Rspace$. Then we have
\begin{equation}\label{eq:asymp_behavior}
\begin{aligned}
&\mathrm{(a)} \, \lambda(t) \sim \ln t, &\quad
&\mathrm{(b)} \, \mathrm{e}^{\lambda(t)} \sim t \ln t, &\quad
&\mathrm{(c)} \, \Herror(t) = \mathcal{O}\left( \frac{(\ln t)^{\Delta-3}}{t^2} \right),&\quad
&\mathrm{(d)} \, \Herrorint(t) = \mathcal{O}\left( \frac{(\ln t)^{\Delta-3}}{t} \right),\\
&\mathrm{(e)} \, u(t) \sim -\frac{t}{\ln t},&\quad
&\mathrm{(f)} \, \werror(t) \sim - \Upsilon_{\underline{r}}\cdot \frac{t}{\ln t}, &\quad
&\mathrm{(g)} \, u'(t) \sim -\frac{1}{\ln t},&\quad
&\mathrm{(h)} \, \werror'(t) \sim -\Upsilon_{\underline{r}} \cdot \frac{1}{\ln t}.
\end{aligned}
\end{equation}
Moreover, the asymptotic relations~(f) and~(h) hold uniformly in $\underline{r} \in \Rspace$.
\end{claim}

\begin{proof}
All of these asymptotic relations follow from the explicit formulas that we have already described (see e.g.\ \eqref{eq:H_error_formula}, \eqref{eq:H_error_int}, \eqref{eq:w_error_ivp_solution})
with the help of standard calculus methods, e.g.\ l'H\^{o}pital's rule. For the sake of completeness, we detail these calculations below.
\begin{enumerate}[label=(\alph*)]
\item First, observe that, using l'H\^{o}pital's rule, we obtain $\lambda(t) \sim \ln t$, because
\begin{equation*}
\lim \limits_{t \to \infty} \frac{t}{\mathrm{e}^{\lambda(t)}} \stackrel{\eqref{eq:lambda_2reg_ode}}{=} \lim \limits_{t \to \infty} \frac{1}{1+\lambda(t)} = 0,
\text{ therefore, }
\lim \limits_{t \to \infty} \frac{\lambda(t)}{\ln t} \stackrel{\eqref{eq:lambda_2reg_ode}}{=} \lim \limits_{t \to \infty} \frac{t(1+\lambda(t))}{\mathrm{e}^{\lambda(t)}} = \lim \limits_{t \to \infty} \frac{\lambda(t) + t \lambda'(t)}{1+\lambda(t)} = 1.
\end{equation*}
\item We use l'H\^{o}pital's and $(\mathrm{a})$ : $\lim \limits_{t \to \infty} \frac{\mathrm{e}^{\lambda(t)}}{t \ln t} \stackrel{\eqref{eq:lambda_2reg_ode}}{=} \lim \limits_{t \to \infty} \frac{1+\lambda(t)}{\ln t + 1} \stackrel{\eqref{eq:asymp_behavior}(\mathrm{a})}{=}1$, i.e., $\mathrm{e}^{\lambda(t)} \sim t \ln t$ holds.

\item In order to study the asymptotic behavior of $\Herror(t)$, we use the formula \eqref{eq:H_error_formula}. Therefore, we need to study the function~$\lamerror(t)$ (see~\eqref{eq:lam_error_def}). Let us introduce the notation $J_{\underline{r}}(t)$ for the integral that appears in the explicit formula of the function~$\lamerror(t)$ (see~\eqref{eq:lam_error_solution}), i.e., $J_{\underline{r}}(t):=\int_0^{t} \frac{1}{1+\lambda(s)} \cdot \left( \sum_{k=2}^{\Delta-1} \frac{\lambda^k(s)}{k!} \cdot s_{k+1}\right) \, \mathrm{d}s$.
Using l'H\^{o}pital's rule and~\eqref{eq:asymp_behavior}(a), it can be shown that
\begin{equation}\label{eq:J_r_t_O}
J_{\underline{r}}(t) = \int_1^{t} \mathcal{O} \left(\ln^{\Delta-2}(s)\right)\, \mathrm{d}s=
\mathcal{O} \left( t \cdot (\ln t)^{\Delta-2}\right).
\end{equation}
Note also that~\eqref{eq:asymp_behavior}(a),~(b) with~\eqref{eq:lambda_2reg_ode} imply that
\begin{equation}\label{eq:lam_diff_asymp}
\lambda'(t) \sim \frac 1t.
\end{equation}

Therefore, we obtain 
\begin{equation}\label{eq:lam_error_asymp}
\lamerror(t) \stackrel{\eqref{eq:lam_error_solution}}{=} \lambda'(t) \cdot J_{\underline{r}}(t) \stackrel{\eqref{eq:J_r_t_O}, \eqref{eq:lam_diff_asymp}}{=} \mathcal{O} \left( (\ln t)^{\Delta-2} \right).
\end{equation}

Moreover, \eqref{eq:lam_error_ode} with \eqref{eq:asymp_behavior}(a),~(b) and \eqref{eq:lam_error_asymp} implies that
\begin{equation}\label{eq:lam_error_diff_asymp}
\lamerror'(t) = \mathcal{O} \left( \frac{(\ln t)^{\Delta - 2}}{t} \right).
\end{equation}
Now we can conclude \eqref{eq:asymp_behavior}(c) by \eqref{eq:H_error_formula}, \eqref{eq:lam_error_asymp}, \eqref{eq:lam_error_diff_asymp}, \eqref{eq:lam_diff_asymp}, \eqref{eq:asymp_behavior}(a) and~(b).

\item We can prove \eqref{eq:asymp_behavior}(d) similarly, using the explicit formula of $\Herrorint(t)$ (see \eqref{eq:H_error_int}).

\item Let us introduce the notation $L(t)$ for the integral that appears in the definition of $u(t)$ (see~\eqref{eq:w_error_ivp_solution_aux_functions}), i.e., $L(t):= \int_0^t \frac{\lambda'(s) - 1}{\lambda^2(s)} \, \mathrm{d}s$. Thus, we have ${u(t) = 1+\lambda(t)L(t)}$. Then using \eqref{eq:asymp_behavior}(a), we can deduce that
\begin{equation}\label{eq:u_asymp}
L(t) \sim -\frac{t}{(\ln t)^2},
\quad \text{therefore,} \quad
u(t) \sim -\frac{t}{\ln t}.
\end{equation}

\item In order to study the asymptotic behavior of $\werror(t)$, we use the formula \eqref{eq:w_error_ivp_solution}. By \eqref{eq:w_error_ivp_solution_aux_functions}, \eqref{eq:asymp_behavior}(a),~(c) and~(e), we obtain that
\begin{equation}\label{eq:werror_first_asymp}
\werrorfirst(t) = \mathcal{O}\left( (\ln t )^{\Delta-2}\right).
\end{equation}

Therefore, \eqref{eq:w_error_2nd_limit}, \eqref{eq:asymp_behavior}(e) and \eqref{eq:werror_first_asymp} imply that $\werrorfirst(t)\lambda(t) \ll \werrorsecond(t)u(t)$ uniformly in $\underline{r} \in \Rspace$, and hence by \eqref{eq:w_error_ivp_solution}, we obtain that~{\eqref{eq:asymp_behavior}(f)} holds uniformly in $\underline{r} \in \Rspace$.

\item The asymptotic relation \eqref{eq:asymp_behavior}(g) follows from \eqref{eq:werror_u_derivative} and \eqref{eq:u_asymp}.

\item Finally, from \eqref{eq:werror_u_derivative}, \eqref{eq:werror_first_asymp}, \eqref{eq:lam_diff_asymp}, \eqref{eq:w_error_2nd_limit} and \eqref{eq:asymp_behavior}(g) we obtain that~\eqref{eq:asymp_behavior}(h) holds uniformly in~${\underline{r} \in \Rspace}$.\qedhere
\end{enumerate}
\end{proof}

Now we are ready to prove Lemma~\ref{lem:monotonicity}.

\begin{proof}[Proof of Lemma~\ref{lem:monotonicity}]
Recall the notion of $\varepsilon$ and $\underline{r}=(r_k)_{k=2}^\Delta$ that corresponds to $\underline{p}$ from~\eqref{eq:eps_rk_def}.
Recall the formula of~$\werror'(t)$ from~\eqref{eq:werror_u_derivative}. In order to prove~\eqref{eq:w_error_derivative_goal}, it is enough to show that there exists~${\varepsilon_0>0}$ such that for any ${\underline{p} \in \nearreg{\varepsilon_0}}$ and ${t \in [0,\toverone]}$ both $\varepsilon \cdot \overline{K} \cdot |\werrorfirst(t)\lambda'(t)|$ and $\varepsilon \cdot \overline{K} \cdot |\werrorsecond(t)u'(t)|$ are less than $\lambda'(t)/2$. First, we show that
\begin{equation}\label{eq:monotonicity_1st_term_goal}
\varepsilon \cdot \overline{K} \cdot |\werrorfirst(t) \lambda'(t)| \le \frac 12 \lambda'(t), 
\quad \text{i.e.,} \quad
\varepsilon \cdot \overline{K} \cdot |\werrorfirst(t)| \le \frac 12.
\end{equation}

Since the function $(\lambda'(t)-1)/\lambda^2(t)$ is bounded and $u(0)=1$, there exists $\widetilde{K} \in \mathbb{R}_+$ such that for any $t \in \mathbb{R}_+$, we have
\begin{equation}\label{eq:u_upper_bound}
|u(t)|
\stackrel{\eqref{eq:w_error_ivp_solution_aux_functions}}{\le}
1 + \lambda(t) \cdot \widetilde{K} \cdot t.
\end{equation}

Therefore, for any $\underline{p} \in \nonreg$, we have
\begin{equation}\label{eq:werrorfirst_upp_bound1}
|\werrorfirst(t)|
\stackrel{\eqref{eq:w_error_ivp_solution_aux_functions}, \, \eqref{eq:u_upper_bound}, \, \eqref{eq:H_error_abs_bound}}{\le}
\int \limits_0^t (1 + \lambda(s) \cdot \widetilde{K} \cdot s) \cdot 2\lambda'(s) \mathrm{e}^{-\lambda(s)} \cdot \sum_{k=0}^{\Delta-1} \frac{\lambda^k(s)}{k!} \cdot \lambda(s) \, \mathrm{d}s.
\end{equation}

Hence, for any $t \in [0, \toverone]$, we obtain
\begin{align}
\begin{split}\label{eq:werrorfirst_upp_bound2}
|\werrorfirst(t)| \stackrel{\eqref{eq:werrorfirst_upp_bound1},\, \eqref{eq:lambda_2reg_ode},\, \eqref{eq:lam_grows_logarithmic}}{\le}
&4\widetilde{K}\Delta \cdot \lambda^{\Delta+2}(\toverone) \int \limits_0^t \frac{1}{(s+1)} \, \mathrm{d}s = \\
&4\widetilde{K}\Delta \cdot \lambda^{\Delta+2}(\toverone) \cdot \ln (t+1) \stackrel{\eqref{eq:lam_grows_logarithmic}}{\le} 4\widetilde{K}\Delta \cdot \lambda^{\Delta+3}(\toverone).
\end{split}
\end{align}

By \eqref{eq:lambda_universal_bounds}(b), \eqref{eq:werrorfirst_upp_bound2} implies \eqref{eq:monotonicity_1st_term_goal} if $\varepsilon_0$ is small enough.

Now let us consider the term $|\werrorsecond(t)u'(t)|$. For any $\underline{p} \in \nonreg$, we have
\begin{equation}\label{eq:monotonicity_2nd_term_bound1}
\frac{|\werrorsecond(t)u'(t)|}{\lambda'(t)} \stackrel{\eqref{eq:w_error_2nd_bound}}{\le} \frac{2\Delta^2(\Delta+1)|u'(t)|}{\lambda'(t)}\stackrel{\eqref{eq:lambda_2reg_ode},\, \eqref{eq:asymp_behavior}}{\sim} 2\Delta^2(\Delta+1) \cdot \frac{t}{\ln t}.
\end{equation}

Therefore, there exists $T \in \mathbb{R}_+$ such that for any $t \in [T, \toverone]$, we have
\begin{equation}\label{eq:monotonicity_2nd_term_bound2}
\varepsilon \cdot \overline{K} \cdot \frac{|\werrorsecond(t)u'(t)|}{\lambda'(t)} \stackrel{\eqref{eq:monotonicity_2nd_term_bound1}}{\le} 4\overline{K}\Delta^2(\Delta+1) \cdot \varepsilon \cdot\frac{\toverone}{\ln (\toverone)} \stackrel{\eqref{eq:tc_approx}, \, \eqref{eq:distribution_with_eps}}{=} \frac{8\overline{K}\Delta^2(\Delta+1)}{\Upsilon_{\underline{r}} \cdot \ln (\toverone)}.
\end{equation}

Note that $\toverone \to \infty$ as $\varepsilon \to 0$ (see \eqref{eq:tc_approx}) and $\Upsilon_{\underline{r}} \ge 3$ for any $\underline{r} \in \Rspace$ (see~\eqref{eq:Upsilon_r_def}). Therefore, \eqref{eq:monotonicity_2nd_term_bound2} shows that if~$\varepsilon$ is small enough, then on the interval ${[T, \toverone]}$, we have ${\varepsilon \cdot \overline{K} \cdot |\werrorsecond(t)u'(t)| \le \lambda'(t)/2}$ . On the other hand, since the function ${2\Delta^2(\Delta+1)|u'(t)|}$ is bounded on the compact interval $[0,T]$, we can choose~$\varepsilon$ in a way that ${\varepsilon \cdot \overline{K} \cdot |\werrorsecond(t)u'(t)| \le \lambda'(t)/2}$ holds also on the interval $[0,T]$. The proof of Lemma~\ref{lem:monotonicity} is complete.
\end{proof}

\subsubsection{Inequalities for linear approximations}\label{sss:inequalities}

In Section~\ref{sss:inequalities} we prove Lemma~\ref{lem:sufficient_inequalities}, which is the final step in the proof of our main result. We introduce the difference functions~$\diffunder$ and~$\diffover$ that measure the gap between the two sides of the inequalities~\eqref{eq:w_under_derivative_goal} and~\eqref{eq:w_over_derivative_goal}, respectively (cf.~Notation~\ref{not:diff_under_over} below). By utilizing the explicit formulas for the functions that appear in the inequalities~\eqref{eq:w_under_derivative_goal}, \eqref{eq:w_over_derivative_goal}, i.e., for $\wunder$, $\wover$ and the survival functions of $\tauunder{\delta/3}$ and $\tauover{\delta/3}$, we can expand the difference functions~$\diffunder$ and~$\diffover$ in terms of~$\varepsilon$. Using that~$\werror$ is defined as the solution of the ODE~\eqref{eq:w_error_ivp}, the leading-order terms in these expansions cancel out. This leaves only residual error terms that can be controlled using the asymptotic relations established in Section~\ref{sss:asymp_prop}.

\begin{notation}[Difference functions]\label{not:diff_under_over}
Let $\delta \in (0, 1)$ and $\underline{p} \in \nonreg$. Let us define
\begin{align}
\diffunder(t) &:= \wunder'(t) - \ev \left( \wunder(\tauover{\delta/3}) \cdot \indic{t < \tauover{\delta/3}} \right), \label{eq:diff_under_def}\\
\diffover(t) &:= \wover'(t) - \ev \left( \wover(\tauunder{\delta/3}) \cdot \indic{t < \tauunder{\delta/3}} \right). \label{eq:diff_over_def}
\end{align}
\end{notation}

In order to prove Lemma~\ref{lem:sufficient_inequalities}, we need to show that $\diffunder(t) \ge 0$ for any $t \in [0,\tunder]$ and $\diffover(t) \le 0$ for any $t \in [0,\tover]$. For that purpose, we use the next lemma.

\begin{lemma}[Properties of the difference functions]\label{lem:diff_functions_prop}
For any $\delta \in (0, 1)$, there exists $\varepsilon_0>0$ such that for any $\underline{p} \in \nearreg{\varepsilon_0}$, we have
\begin{align}
\mathrm{(a)}\quad
\diffunder(\tunder) &\ge 0,
\qquad\qquad\quad
\mathrm{(b)}\quad
\int \limits_t^{\tunder} \diffunder'(s) \, \mathrm{d}s \le 0
\quad \text{for any} \quad t \in [0,\tunder],\label{eq:difference_under_prop}\\
\mathrm{(a)}\quad
\diffover(\tover) &\le 0,
\qquad\qquad\quad
\mathrm{(b)}\quad
\int \limits_t^{\tover} \diffover'(s) \, \mathrm{d}s \ge 0
\quad \text{for any} \quad t \in [0,\tover].\label{eq:difference_over_prop}
\end{align}
\end{lemma}

Before proving Lemma~\ref{lem:diff_functions_prop}, we first derive Lemma~\ref{lem:sufficient_inequalities} from it.

\begin{proof}[Proof of Lemma~\ref{lem:sufficient_inequalities}]
Let $\delta \in (0, 1)$ and $\varepsilon_0>0$ given by Lemma~\ref{lem:diff_functions_prop}. For any $\underline{p} \in \nearreg{\varepsilon_0}$, we have
\begin{align}
\diffunder(t) &= \diffunder(\tunder) - \int \limits_t^{\tunder} \diffunder'(s) \, \mathrm{d}s \stackrel{\eqref{eq:difference_under_prop}}{\ge} 0
\quad \text{for any} \quad t \in [0,\tunder], \label{eq:diff_under_nonneg} \\
\diffover(t) &= \diffover(\tover) - \int \limits_t^{\tover} \diffover'(s) \, \mathrm{d}s \stackrel{\eqref{eq:difference_over_prop}}{\le} 0
\quad \text{for any} \quad t \in [0,\tover] \label{eq:diff_over_nonpos}.
\end{align}

By the definition of~$\diffunder(t)$ (see~\eqref{eq:diff_under_def}), inequality~\eqref{eq:diff_under_nonneg} is equivalent to~\eqref{eq:w_under_derivative_goal}, and by the definition of~$\diffover(t)$ (see~\eqref{eq:diff_over_def}), inequality~\eqref{eq:diff_over_nonpos} is equivalent to~\eqref{eq:w_over_derivative_goal}. The proof of Lemma~\ref{lem:sufficient_inequalities} is complete.
\end{proof}

Therefore, it remains to prove Lemma~\ref{lem:diff_functions_prop}. For that purpose, we will use the following two statements. In particular, the limits stated in Claim~\ref{cl:diff_functions_1st_lemma} will ensure that the inequalities~\eqref{eq:difference_under_prop}(a) and~\eqref{eq:difference_over_prop}(a) are satisfied, and the integral inequalities established in Lemma~\ref{lem:diff_functions_2nd_lemma} will turn out to be useful when we prove the inequalities~\eqref{eq:difference_under_prop}(b) and~\eqref{eq:difference_over_prop}(b).

Recall the notation~$\Upsilon_{\underline{r}}$, $\Herrorint(t)$ and~$\werror(t)$ from equation~\eqref{eq:Upsilon_r_def}, Definitions~\ref{def:H_approx} and~\ref{def:w_approx}, respectively.

\begin{claim}[Auxiliary limits]\label{cl:diff_functions_1st_lemma}
The following limits hold uniformly in~$\underline{r} \in \Rspace$:
\begin{align}
&\lim \limits_{t \to \infty} \frac 1t \cdot \left( \werror(t) - \mathrm{e}^{\lambda(t)} \cdot \werror'(t)\right) =\Upsilon_{\underline{r}},\label{eq:aux_limit}\\
&\lim \limits_{t \to \infty} \frac{1}{t} \cdot \mathrm{e}^{\lambda(t)} \cdot \lambda(t) \cdot \Herrorint(t) = 0,\label{eq:aux_limit2}\\
&\lim \limits_{t \to \infty} \frac{1}{t^2} \cdot \mathrm{e}^{\lambda(t)} \cdot \Herrorint(t) \cdot \werror(t) = 0.\label{eq:aux_limit3}
\end{align}
\end{claim}

\begin{proof}
The statements follow from the asymptotic relations stated in Claim~\ref{cl:asymp}.
\end{proof}

Recall the function $\Herror(t)$ from Definition~\ref{def:H_approx}.

\begin{lemma}[Integral of an error term]\label{lem:diff_functions_2nd_lemma}
For any $\delta \in (0,1)$ and $K \in \mathbb{R}$, there exists $\varepsilon_0>0$ such that for any $\underline{p} \in \nearreg{\varepsilon_0}$, we have
\begin{align}
&\int \limits_t^{\tunder} \left[\lambda(s)+\varepsilon K \werror(s) \right] \cdot \Herror(s) \, \mathrm{d}s \ge 0
\quad \text{for any} \quad t \in [0,\tunder],\label{eq:diff_aux_int_tunder_pos}\\
&\int \limits_t^{\tover} \left[\lambda(s)+\varepsilon K \werror(s) \right] \cdot \Herror(s) \, \mathrm{d}s \ge 0
\quad \text{for any} \quad t \in [0,\tover],\label{eq:diff_aux_int_tover_pos}
\end{align}
where the $\varepsilon$ and the $\underline{r} = (r_k)_{k=2}^\Delta$ corresponding to $\underline{p}$ are defined in~\eqref{eq:eps_rk_def}.
\end{lemma}

In the proof of Lemma~\ref{lem:diff_functions_2nd_lemma} we will treat the $t\leq \theta$ and the $t>\theta$ cases separately, where we define~$\theta$ in the following statement. 
Lemma~\ref{lem:diff_functions_2nd_lemma} will turn out to be an immediate consequence of Lemma~\ref{lem:monotonicity} if $t>\theta$.

\begin{claim}[Positivity of $\Herror$]\label{cl:Herror_positivity}
There exists $\theta \in \mathbb{R}_+$ such that $\Herror(t)>0$ for any $\underline{r} \in \Rspace$ and $t \ge \theta$.
\end{claim}

\begin{proof}
Observe that $\Herror(t)$ is linear in~$\underline{r}$ (cf.~\eqref{eq:H_error_formula}, \eqref{eq:lam_error_ode}, \eqref{eq:lam_error_solution}, \eqref{eq:sk_def}). For any $d \in \{3, \dots, \Delta\}$, we define $\underline{r}^{(d)} \in \Rspace$ as
\begin{equation*}
\underline{r}^{(d)} = (r^{(d)}_k)_{k=2}^{\Delta}
\quad \text{such that} \quad
r^{(d)}_2 = -1, \; r^{(d)}_d = 1 \text{ and } r^{(d)}_k = 0 \text{ for any } k \notin\{2,d\}.
\end{equation*}

Then by \eqref{eq:R_space} for any $\underline{r} \in \Rspace$, we have $\underline{r} = \sum_{d=3}^{\Delta} r_d \cdot \underline{r}^{(d)}$. Therefore, we can decompose $\Herror(t)$ as
\begin{equation}\label{eq:Herror_decomposition}
\Herror(t) = \sum \limits_{d=3}^{\Delta} r_d \cdot \mathcal{H}_{\underline{r}^{(d)}}(t).
\end{equation}

Using~\eqref{eq:H_error_formula}, \eqref{eq:lam_error_ode} and \eqref{eq:lam_error_bounds}, it can be shown that the leading order term of~$\mathcal{H}_{\underline{r}^{(d)}}(t)$ has positive coefficient for any $d \in \{3, \dots, \Delta\}$. Thus, for each fixed $d \in \{3, \dots, \Delta\}$, there exists $\theta_d \in \mathbb{R}_+$ such that $\mathcal{H}_{\underline{r}^{(d)}}(t) > 0$ for any $t > \theta_d$. By the decomposition~\eqref{eq:Herror_decomposition}, $\theta := \max \{\,\theta_d \;|\; d = 3, \dots, \Delta\}$ satisfies the statement.
\end{proof}

Now we are ready to prove Lemma~\ref{lem:diff_functions_2nd_lemma}.

\begin{proof}[Proof of Lemma~\ref{lem:diff_functions_2nd_lemma}]
Fix $\delta \in (0,1)$ and ${K \in \mathbb{R}}$. We prove \eqref{eq:diff_aux_int_tunder_pos}, the proof of~\eqref{eq:diff_aux_int_tover_pos} is analogous. Let us define the function ${w_{\underline{p},K}(t) := \lambda(t)+\varepsilon K \werror(t)}$, where $\underline{p} \in \nonreg$ and the $\varepsilon$ and the $\underline{r} = (r_k)_{k=2}^\Delta$ corresponding to $\underline{p}$ are defined in~\eqref{eq:eps_rk_def}. Thus,~\eqref{eq:diff_aux_int_tunder_pos} is equivalent to 
\begin{equation}\label{eq:w_K_Herror_int}
\int \limits_t^{\tunder} w_{\underline{p},K}(s) \Herror(s) \, \mathrm{d}s \ge 0
\qquad \text{for any} \qquad
t \in [0,\tunder].
\end{equation}

In order to prove~\eqref{eq:w_K_Herror_int}, we use Claim~\ref{cl:Herror_positivity} that provides $\theta > 0$ such that $\Herror(s) > 0$ for any $\underline{r} \in \Rspace$ and~${s \ge \theta}$. Note that we can assume that ${\tunder > 2\theta}$, since if $\varepsilon_0$ is small enough, then it holds for any ${\underline{p} \in \nearreg{\varepsilon_0}}$.
We split our analysis into two cases based on the value of~$t$. In both cases we apply Lemma~\ref{lem:monotonicity} and we set~${\overline{K}:=2|K|}$.

\begin{itemize}
\item If $t \in (\theta, \tunder]$:

By Lemma~\ref{lem:monotonicity} (with $\overline{K}=2|K|$), if $\varepsilon_0$ is small enough, then for any $\underline{p} \in \nearreg{\varepsilon_0}$, the function~$w_{\underline{p},K}(\cdot)$ is non-negative on the interval~$[0,\toverone]$. Therefore, both $w_{\underline{p},K}(\cdot)$ and $\Herror(\cdot)$ are non-negative on the interval~${[t, \tunder]}$, hence~\eqref{eq:w_K_Herror_int} trivially holds.

\item If $t \in [0,\theta]$:

We express the integral in~\eqref{eq:w_K_Herror_int} using integration by parts, noting that $\Herror(s) = -\Herrorint'(s)$ (cf.~\eqref{eq:H_error_def}):
\begin{equation}\label{eq:wK_Herror_int_partial}
\int \limits_t^{\tunder}w_{\underline{p},K}(s) \Herror(s) \, \mathrm{d}s = w_{\underline{p},K}(t) \Herrorint(t) - w_{\underline{p},K}(\tunder) \Herrorint(\tunder) + \int \limits_t^{\tunder} w_{\underline{p},K}'(s) \Herrorint(s) \, \mathrm{d}s.
\end{equation}

We study the three terms on the right-hand side of~\eqref{eq:wK_Herror_int_partial} separately. We show that the first term is non-negative, the second term vanishes as $\varepsilon\to 0$ and the third term is positive and bounded away from~0, i.e., the sum is strictly positive if $\varepsilon$ is small enough.

\begin{enumerate}
\item By Claim~\ref{cl:H_error_explicit} and Lemma~\ref{lem:monotonicity} (with $\overline{K}=2|K|$), we obtain that the first term on the right-hand side of~\eqref{eq:wK_Herror_int_partial} is non-negative for any $\underline{p} \in \nearreg{\varepsilon_0}$ if~$\varepsilon_0$ is small enough.

\item By Claim~\ref{cl:asymp}, the second term on the right-hand side of~\eqref{eq:wK_Herror_int_partial} converges to $0$ as ${\varepsilon \to 0}$.

\item If $\varepsilon$ is small enough, then for any~${\underline{r} \in \Rspace}$ and $t \in [0,\toverone]$, we obtain
\begin{equation}\label{eq:third_term_lower_bound}
\Herrorint(s) \stackrel{\eqref{eq:H_error_int},\, \eqref{eq:lam_error_bounds}}{\ge} \frac 12 \mathrm{e}^{-\lambda(s)} \cdot \lambda(s)
\quad \text{and} \quad
w'_{\pfromr,K}(t) \stackrel{\eqref{eq:w_error_derivative_goal}}{\ge} \frac{\lambda'(t)}{2},
\end{equation}
where~$\pfromr$ is defined in Notation~\ref{not:p(eps,r)} and~\eqref{eq:w_error_derivative_goal} is used with $\overline{K}=2|K|$. Therefore, the third term on the right-hand side of~\eqref{eq:wK_Herror_int_partial} is strictly positive and bounded away from 0, in particular
\begin{equation*}
\int \limits_t^{\tunder} w_{\underline{p},K}'(s) \Herrorint(s) \, \mathrm{d}s 
\stackrel{\eqref{eq:third_term_lower_bound}}{\ge}
\frac14 \int \limits_\theta^{2\theta} \lambda'(s) \cdot \mathrm{e}^{-\lambda(s)} \cdot \lambda(s) \, \mathrm{d}s.
\end{equation*}
\end{enumerate}
Thus, we can conclude that~\eqref{eq:w_K_Herror_int} holds for any~${t \in [0,\theta]}$ if $\varepsilon$ is small enough.
\end{itemize}

Therefore, the proof of~\eqref{eq:diff_aux_int_tunder_pos} is complete.
\end{proof}

Now we can conclude Lemma~\ref{lem:diff_functions_prop}.
We substitute the explicit first-order Taylor approximations (cf.~\eqref{eq:H_int_under_over_def}, \eqref{eq:w_approx}) into the definitions of the difference functions $\diffunder$, $\diffover$ and expand them in terms of the small parameter~$\varepsilon$. Then, the fact that the desired inequalities are satisfied at the right endpoint of the given interval (i.e.,~\eqref{eq:difference_under_prop}(a) and~\eqref{eq:difference_over_prop}(a) are fulfilled) follows from the choice of~$\tunder$ and $\tover$ and the limits stated in Claim~\ref{cl:diff_functions_1st_lemma}. To establish \eqref{eq:difference_under_prop}(b) and~\eqref{eq:difference_over_prop}(b), we exploit the fact that~$\werror$ satisfies the ODE~\eqref{eq:w_error_ivp}, which causes some terms to cancel. The inequalities then follow from Lemma~\ref{lem:diff_functions_2nd_lemma}.

\needspace{3\baselineskip}
\begin{proof}[Proof of Lemma~\ref{lem:diff_functions_prop}]
First, we show \eqref{eq:difference_under_prop}. The proof of~\eqref{eq:difference_over_prop} is analogous to the proof of~\eqref{eq:difference_under_prop}.
\begin{itemize}
\item Proof of~\eqref{eq:difference_under_prop}:

First, we substitute the definitions of~$\wunder$ and $\tauover{\delta/3}$ into the definition of~$\diffunder$ to obtain a formula that we use in the proof of both inequalities of~\eqref{eq:difference_under_prop}. In particular, for any ${t \in [0,\tunder]}$, we have
\begin{equation}
\label{eq:diff_under_rewrite}
\diffunder(t) \stackrel{\eqref{eq:diff_under_def},\, \eqref{eq:tau_pm_def},\,\eqref{eq:w_approx}}{=}
\wunder'(t) - \int \limits_t^{\tunder}\wunder(s) \cdot \left(-\Hintover{\delta/3}'(s) \right) \, \mathrm{d}s -
\Hintover{\delta/3}(\tunder) \cdot \wunder(\tunder).
\end{equation}

\begin{enumerate}[label=(\alph*)]
\item In order to prove the inequality~\eqref{eq:difference_under_prop}(a), we substitute $t=\tunder$ into~\eqref{eq:diff_under_rewrite} and the explicit formulas for~$\wunder$, $\Hintover{\delta/3}$ and expand it in terms of~$\varepsilon$:
\begin{align}
\begin{split}\label{eq:diff_under_a_with_eps}
\diffunder(\tunder) \stackrel{\eqref{eq:diff_under_rewrite}}{=}& \wunder'(\tunder) -
\Hintover{\delta/3}(\tunder) \cdot \wunder(\tunder)
\stackrel{\eqref{eq:w_approx}, \, \eqref{eq:H_int_under_over_def}}{=}\\
&\lambda'(\tunder) +\left(1+\frac{\delta}{2}\right)\varepsilon\werror'(\tunder) -\\
&\left[ I(\tunder) + \left(1+\frac{\delta}{3}\right)\varepsilon \cdot \Herrorint(\tunder) \right]
\cdot
\left[ \lambda(\tunder) +\left(1+\frac{\delta}{2}\right)\varepsilon\werror(\tunder) \right]
\stackrel{\eqref{eq:lambda_2reg_ode},\, \eqref{eq:I(t)_explicit}}{=}\\
&\mathrm{e}^{-\lambda(\tunder)}+
\varepsilon \cdot \left(1+\frac{\delta}{2}\right) \cdot \left[\werror'(\tunder)-
\mathrm{e}^{-\lambda(\tunder)} \werror(\tunder)\right] -\\
&\varepsilon \cdot \left(1+\frac{\delta}{3} \right) \cdot \Herrorint(\tunder)\lambda(\tunder) -\varepsilon^2 \cdot \left(1+\frac{\delta}{2}\right) \left(1+\frac{\delta}{3} \right) \cdot \Herrorint(\tunder)\werror(\tunder).
\end{split}
\end{align}

If we multiply the above formula by~$\exp(\lambda(\tunder))$, we obtain that~\eqref{eq:difference_under_prop}(a) is equivalent to
\begin{align}
\begin{split}\label{eq:diff_under_1st_goal}
1+&\varepsilon \cdot \left(1+\frac{\delta}{2}\right) \cdot \left(\mathrm{e}^{\lambda(\tunder)}\werror'(\tunder)-\werror(\tunder)\right) - 
\varepsilon \cdot \left(1+\frac{\delta}{3} \right) \cdot \mathrm{e}^{\lambda(\tunder)}\Herrorint(\tunder)\lambda(\tunder) -\\
&\varepsilon^2 \cdot \left(1+\frac{\delta}{2}\right) \cdot \left(1+\frac{\delta}{3} \right) \cdot \mathrm{e}^{\lambda(\tunder)}\Herrorint(\tunder)\werror(\tunder) \ge 0.
\end{split}
\end{align}

Therefore, it is enough to see that the limit of the left-hand side of \eqref{eq:diff_under_1st_goal} as $\varepsilon \to 0$ is strictly positive. We show that the limit of the sum of the first two terms is strictly positive and the other two terms vanish as $\varepsilon \to 0$. For that purpose, we use the limits stated in Claim~\ref{cl:diff_functions_1st_lemma}.

\begin{itemize}
\item On one hand, if we consider the first two terms of the left-hand side of \eqref{eq:diff_under_1st_goal}, we obtain
\begin{align}
\begin{split}\label{eq:diff_under_limit}
\lim \limits_{\varepsilon \to 0} 1+\varepsilon \cdot \left(1+\frac{\delta}{2}\right) \cdot \left(\mathrm{e}^{\lambda(\tunder)}\werror'(\tunder)-\werror(\tunder)\right) \stackrel{\eqref{eq:aux_limit}}{=}&
1-\left(1+\frac{\delta}{2}\right) \cdot \lim \limits_{\varepsilon \to 0} \, \varepsilon \cdot \Upsilon_{\underline{r}} \cdot \tunder \stackrel{\eqref{eq:tc_approx}}{=}\\
&1-\left(1+\frac{\delta}{2}\right)(1-\delta) = 
\frac{\delta}{2}+\frac{\delta^2}{2},
\end{split}
\end{align}
which is strictly positive for any $\delta \in (0,1)$.

\item On the other hand, by the definition of $\tunder$ (see~\eqref{eq:tc_approx}), equations \eqref{eq:aux_limit2} and \eqref{eq:aux_limit3}, we obtain that the last two terms of the left-hand side of \eqref{eq:diff_under_1st_goal} converge to $0$ as $\varepsilon \to 0$.
\end{itemize}

Thus, \eqref{eq:diff_under_1st_goal} holds if $\varepsilon$ is small enough, therefore, the proof of~\eqref{eq:difference_under_prop}(a) is complete.

\item Now we prove the inequality \eqref{eq:difference_under_prop}(b). By differentiating the formula obtained in~\eqref{eq:diff_under_rewrite}, substituting the explicit formulas of~$\wunder$, $\Hintover{\delta/3}$, using that~$\werror$ solves the ODE~\eqref{eq:w_error_ivp}, and expanding the formula in terms of~$\varepsilon$, we obtain that for any $t \in [0,\tunder]$, we have
\begin{align}
\begin{split}\label{eq:diff_under_derivative}
\diffunder'(t) \stackrel{\eqref{eq:diff_under_rewrite}}{=}& \wunder''(t) + \wunder(t) \cdot \left(-\Hintover{\delta/3}'(t) \right)\stackrel{\eqref{eq:w_approx}, \, \eqref{eq:H_int_under_over_def}}{=}\lambda''(t) +\left(1+\frac{\delta}{2}\right)\varepsilon\werror''(t) +\\
&\left[ \lambda(t) +\left(1+\frac{\delta}{2}\right)\varepsilon\werror(t)\right] \cdot \left[H(t) + \left(1+\frac{\delta}{3}\right)\varepsilon \cdot \Herror(t)\right]\stackrel{\substack{\eqref{eq:lambda_2reg_ode},\, \eqref{eq:I(t)_explicit},\\ \eqref{eq:w_error_ivp}}}{=}\\
&-\varepsilon \cdot \frac{\delta}{6} \cdot \lambda(t)\Herror(t)
+\varepsilon^2 \cdot \left(1+\frac{\delta}{2}\right)\left(1+\frac{\delta}{3}\right) \cdot \werror(t) \Herror(t).
\end{split}
\end{align}

Now~\eqref{eq:difference_under_prop}(b) follows from~\eqref{eq:diff_under_derivative} and the inequality~\eqref{eq:diff_aux_int_tunder_pos} with $K=-\frac{(1+{\delta}/{2})(1+{\delta}/{3})}{\delta/6}$.
\end{enumerate}

\item Proof of~\eqref{eq:difference_over_prop}:

\begin{enumerate}[label=(\alph*)]
\item In order to prove the inequality~\eqref{eq:difference_over_prop}(a), similarly to~\eqref{eq:diff_under_a_with_eps}, we can derive that
\begin{align}
\begin{split}\label{eq:diff_over_rewrite}
\diffover(\tover) =&\mathrm{e}^{-\lambda(\tover)}+
\varepsilon \cdot \left(1-\frac{\delta}{2}\right) \cdot \left[\werror'(\tover)-
\mathrm{e}^{-\lambda(\tover)} \werror(\tover)\right] -\\
&\varepsilon \cdot \left(1-\frac{\delta}{3} \right) \cdot \Herrorint(\tover)\lambda(\tover) -\varepsilon^2 \cdot \left(1-\frac{\delta}{2}\right) \left(1-\frac{\delta}{3} \right) \cdot \Herrorint(\tover)\werror(\tover).
\end{split}
\end{align}
After scaling~\eqref{eq:diff_over_rewrite} by~$\exp(\lambda(\tover))$ and taking the limit $\varepsilon \to 0$, we obtain that 
\begin{itemize}
\item the sum of the main terms (i.e., the first two terms) converges to~${-\delta/2+\delta^2/2}$ (using~\eqref{eq:aux_limit}, analogously to~\eqref{eq:diff_under_limit}), which is strictly negative for any $\delta \in (0,1)$,
\item the other terms converge to $0$ (cf.~\eqref{eq:aux_limit2} and~\eqref{eq:aux_limit3}).
\end{itemize}
This completes the proof of~\eqref{eq:difference_over_prop}(a).

\item The inequality~\eqref{eq:difference_over_prop}(b) can be concluded similarly to~\eqref{eq:difference_under_prop}(b) by observing that (analogously to~\eqref{eq:diff_under_derivative}) for any $t \in [0,\tover]$, we have
\begin{equation*}
\diffover'(t) = \varepsilon \cdot \frac{\delta}{6} \cdot \lambda(t)\Herror(t)
+\varepsilon^2 \cdot \left(1-\frac{\delta}{2}\right)\left(1-\frac{\delta}{3}\right) \cdot \werror(t) \Herror(t)
\end{equation*}
and then using~\eqref{eq:diff_aux_int_tover_pos} with $K=\frac{(1-{\delta}/{2})(1-{\delta}/{3})}{\delta/6}$.
\end{enumerate}
\end{itemize}
Thus the proof of Lemma~\ref{lem:diff_functions_prop} is complete, which completes the proof of Lemma~\ref{lem:sufficient_inequalities}, which completes the proof of Lemma~\ref{lem:eigenfunction_bounds}, which in turn completes the proof of our main result Theorem~\ref{thm:tc_asymp_behavior_cont}.
\end{proof}

\textbf{\large Acknowledgments:}
The authors were partially supported by the grant NKFI-FK-142124 of NKFI (National Research, Development and Innovation Office), and by the ERC Synergy Grant No.\ 810115 - DYNASNET.

\bibliographystyle{plain}

\end{document}